\documentclass[11pt]{amsart}
\usepackage{amsfonts,amsmath,amssymb,latexsym}
\usepackage{eucal}
\usepackage{geometry}

\newtheorem{proposition}{Proposition}
\newtheorem{theorem}{Theorem}
\newtheorem{lemma}{Lemma}
\begin{document}

\author{J. H. Lira, G. A. Wanderley}
\title{Mean curvature flow of Killing graphs}
\maketitle

\begin{abstract}
We study a Neumann problem related to the evolution of graphs under mean curvature flow in Riemannian manifolds endowed with a Killing vector field. We prove that in a particular case these graphs converge to a bounded minimal graph which contacts the cylinder over the domain  orthogonally along its boundary.
\end{abstract}

\maketitle

\section{Introduction}

 Let $M$ be a $(n+1)$-dimensional Riemannian manifold endowed with a Killing vector field $Y$. Suppose that the distribution orthogonal to $Y$ is of constant rank and integrable. Given an integral leaf $P$ of that distribution, let $\Omega\subset P$ be a bounded domain with regular boundary $\Gamma =\partial\Omega$. Let $\vartheta: \mathbb{I}\times \bar\Omega \to M$ be the flow generated by $Y$ with initial values in $M$, where $\mathbb{I}$ is a maximal interval of definition. In geometric terms, the ambient manifold is a warped product $M = P\times_{1/\sqrt{\gamma}} \mathbb{I}$ where $\gamma = 1/|Y|^2$.

Given $T\in [0,+\infty)$, let  $u:\bar\Omega\times [0,T)\to \mathbb{I}$ be a smooth function.
Fixed this notation, the Killing graph of $u(\cdot, t),$ $t\in [0,T)$,  is the hypersurface $\Sigma_t \subset M$ parametrized by the map
\[
X(t,x)=\vartheta(u(x,t),x), \quad x\in\bar\Omega.
\]
Notice that this definition could be slightly more general if we suppose that the coordinates of $x\in \bar\Omega$ change with the parameter $t\in [0,T)$.
To abolish this possibility is equivalent to rule out tangential diffeomorphisms of $\Omega$.

The Killing cylinder $K$ over $\Gamma$ is by its turn defined by
\begin{equation}
K=\{\vartheta(s,x): s\in \mathbb{I}, \, x \in \Gamma\}.
\end{equation}

Let $N$ be a unit normal vector field along $\Sigma_t$. In what follows, we denote by $H$ the mean curvature of $\Sigma_t$ with respect to the orientation given by $N$.
We are then concerned with establishing conditions for longtime existence of a prescribed mean curvature flow of the form
\begin{eqnarray}
\label{derX}
& & \frac{\partial X}{\partial t} = (nH-\mathcal{H})N,\\
\label{initial} & & X(0,\cdot)= \vartheta(u_0(\cdot),\cdot),
\end{eqnarray}
for given functions $u_0:\bar\Omega \to \mathbb{R}$ and $\mathcal{H}:\bar\Omega\to\mathbb{R}$.
In order to define boundary conditions for the evolution problem (\ref{derX}) we consider
 a function $\phi\in C^{\infty}(\Gamma)$ such that  $|\phi|\leq\phi_0<1$ for some positive constant $\phi_0$.  Let $\nu$ be the inward unit normal vector field along $K$.  We impose the following Neumann condition associated to (\ref{derX})
\begin{equation}
\label{neumann}
\langle N, \nu\rangle|_{\partial\Sigma_t} =\phi,
\end{equation}
where $\langle \cdot, \cdot \rangle$ denotes the Riemannian metric in $M$.

The main result in this paper may be stated as follows

\begin{theorem}\label{main}
There exists a unique solution $u:\bar\Omega\times [0,\infty)\to\mathbb{I}$ to the problem {\rm (\ref{derX})-(\ref{neumann})}.
Moreover, if $\phi=0$ and $\mathcal{H}=0$ the graphs $\Sigma_t$ converge to a minimal graph which contacts the cylinder $K$ orthogonally along its boundary.
\end{theorem}

Theorem \ref{main} extends Theorem 1.1 in \cite{huisken} as well as Theorem 2.4 in \cite{guan} and Theorem 2.4 in \cite{calle} in a twofold way. The corresponding theorems in \cite{huisken} and \cite{guan} concern evolution of graphs in Euclidean space whereas \cite{calle} deals with the case of graphs in Riemannian product spaces of the form $P\times\mathbb{R}$. Moreover those earlier results hold only for the case when the prescribed mean curvature is $\mathcal{H}=0$. Some related results may be also found in \cite{miquel} and \cite{gauss}.

The paper is organized as follows.
Section \ref{prelim} describes the evolution problem in nonparametric terms. Height and boundary gradient \emph{a priori} estimates for  (\ref{derX})-(\ref{neumann}) are presented respectively  in sections \ref{section-height} and \ref{section-boundary}.  Interior gradient estimates are obtained in Section \ref{section-interior}.  Some technical computations needed in the body of the proofs are collected in an appendix.

In Section \ref{section-asymptotic} we prove the following result about the asymptotic behavior of the mean curvature flow (\ref{derX}) for general $\phi$ and $\mathcal{H}$.

\begin{theorem} Suppose that there exists a solution $v\in C^\infty(\bar\Omega)$ of the elliptic Neumann problem
\begin{eqnarray}
\label{elliptic}& & {\rm div}\frac{\nabla v}{W}-\gamma\langle \bar\nabla_Y Y, \frac{\nabla v}{W} \rangle=\mathcal{H}\quad \textrm{in}\quad\, \Omega\\
\label{neumann-1}& & \langle \nu,N\rangle=\phi(x) \quad\quad \quad\quad \quad\quad\textrm{on}\quad \partial\Omega.
\end{eqnarray}
Then the mean curvature flow {\rm (\ref{derX})-(\ref{neumann})} converges to a graph with prescribed mean curvature $\mathcal{H}$ and prescribed contact angle $\phi$.
\end{theorem}

In a forthcoming paper the authors prove an existence result for (\ref{elliptic})-(\ref{neumann-1}). We also consider there another stationary regimen of the mean curvature flow (\ref{derX})-(\ref{neumann}), namely translating solitons.

\section{Fundamental equations}\label{prelim}

Since we will consider the mean curvature flow in nonparametric terms it seems adequate to describe all geometric invariants as well as their evolution equations in terms of graphical coordinates.

Let $x^1,\ldots, x^n$ be local coordinates in $P$. This system is augmented to be a coordinate system in $M$ by setting $x^0=s$, the flow parameter of $Y$.
The tangent space of $\Sigma_t$ at a point $X(t,x), \, x\in \bar\Omega$, is spanned by the coordinate vector fields
\begin{eqnarray}
\label{coord-graphs}
X_*\frac{\partial }{\partial x^i} =\vartheta_*\frac{\partial}{\partial x^i}+ u_i\vartheta_*\frac{\partial}{\partial x^0}= \frac{\partial}{\partial x^i}\Big|_X + u_i \frac{\partial}{\partial x^0}\Big|_X.
\end{eqnarray}
In terms of these coordinates the induced metric in $\Sigma_t$ is expressed in local components by
\begin{equation}
\label{1st}
g_{ij} = \sigma_{ij}+\frac{1}{\gamma}u_i u_j,
\end{equation}
where $\gamma = \frac{1}{|Y|^2}$ and $\sigma_{ij}$ are the local components of the metric in $P$.

In order to compute the mean curvature of $\Sigma_t$, we fix $N$ as the vector field
\begin{equation}
N = \frac{1}{W}\big(\gamma Y - \vartheta_* \nabla u\big),
\end{equation}
where $\nabla u$ is the gradient of $u$ in $P$ and
\begin{equation}
W = \sqrt{\gamma +|\nabla u|^2}.
\end{equation}
The second fundamental form of $\Sigma_t$ calculated with respect to this choice of normal vector field has local components
\begin{equation}
a_{ij} =\langle \bar\nabla_{X_*\frac{\partial}{\partial x^i}}X_*\frac{\partial}{\partial x^j}, N\rangle,
\end{equation}
where $\bar\nabla$ denotes the covariant derivative in $M$. We then compute
\begin{eqnarray*}
a_{ij} & = & \langle \bar\nabla_{X_*\frac{\partial}{\partial x^i}}\vartheta_*\frac{\partial}{\partial x^j}, N\rangle + \langle \bar\nabla_{X_*\frac{\partial}{\partial x^i}}u_j\vartheta_*\frac{\partial}{\partial x^0}, N\rangle\\
& = & \langle \bar\nabla_{\vartheta_*\frac{\partial}{\partial x^i}}\vartheta_*\frac{\partial}{\partial x^j}, N\rangle + u_i\langle\bar\nabla_{\vartheta_*\frac{\partial}{\partial x^0}}\vartheta_*\frac{\partial}{\partial x^j}, N\rangle +  u_j\langle \bar\nabla_{\vartheta_*\frac{\partial}{\partial x^i}}\vartheta_*\frac{\partial}{\partial x^0}, N\rangle\\
& & \,\,\,\, + u_{i,j}\langle\vartheta_*\frac{\partial}{\partial x^0}, N\rangle + u_i u_j \langle \bar\nabla_{\vartheta_*\frac{\partial}{\partial x^0}}\vartheta_*\frac{\partial}{\partial x^0},N\rangle.
\end{eqnarray*}
Hence using the fact that the maps $x\mapsto \vartheta(s,x)$ are isometries and that the hypersurfaces defined by $\{\vartheta(s,x): x\in P\}, \, s\in \mathbb{I},$ are totally geodesic one concludes that
\begin{eqnarray*}
& & a_{ij}  =  \langle \bar\nabla_{\frac{\partial}{\partial x^i}}\frac{\partial}{\partial x^j}, -\frac{1}{W}\nabla u\rangle + u_i\langle\bar\nabla_{\frac{\partial}{\partial x^j}}Y,\frac{1}{W} \gamma Y\rangle +  u_j\langle \bar\nabla_{\frac{\partial}{\partial x^i}}Y, \frac{1}{W}\gamma Y\rangle\\
& & \,\,\,\, + u_{i,j}\langle Y, \frac{1}{W}\gamma Y\rangle + u_i u_j \langle \bar\nabla_{Y}Y,-\frac{1}{W}\nabla u\rangle.
\end{eqnarray*}
It follows from Killing's equation that
\begin{equation}
\label{aij-gamma}
 a_{ij}  =  \frac{u_{i;j}}{W} - \frac{u_i}{W}\frac{\gamma_j}{2\gamma} - \frac{u_j}{W}\frac{\gamma_i}{2\gamma} -\frac{u_i u_j}{2W} u^k\frac{\gamma_k}{\gamma^2}.
\end{equation}
It turns out that $a_{ij}$ could be also expressed by
\begin{equation}
\label{2nd}
a_{ij} = \frac{u_{i;j}}{W} - \frac{u_i}{W} \gamma\langle \bar\nabla_Y Y, \frac{\partial}{\partial x^j}\rangle - \frac{u_j}{W} \gamma\langle \bar\nabla_Y Y, \frac{\partial}{\partial x^i}\rangle - \frac{u_i u_j}{W} \langle \bar\nabla_{Y}Y,\nabla u\rangle.
\end{equation}
Taking traces with respect to the induced metric one obtains the following expression for the mean curvature $H$ of the hypersurface $\Sigma_t$
\begin{equation}
\label{Hnonpar}
nH = \Big(\sigma^{ij}-\frac{u^i}{W}\frac{u^j}{W}\Big)\frac{u_{i;j}}{W} -\frac{2\gamma+|\nabla u|^2}{W^3}\langle \frac{\bar\nabla \gamma}{2\gamma}, \nabla u\rangle.
\end{equation}
Alternatively one has
\begin{equation}
nH = \Big(\sigma^{ij}-\frac{u^i}{W}\frac{u^j}{W}\Big)\frac{u_{i;j}}{W}-\frac{2\gamma+|\nabla u|^2}{W^3}\gamma\langle \bar\nabla_Y Y, \nabla u\rangle.
\end{equation}
At this point we recall that
\begin{equation}
\label{killing-1}
\bar\nabla_{\frac{\partial}{\partial x^i}}Y = -\frac{1}{2}\frac{\gamma_i}{\gamma}Y
\end{equation}
and
\begin{equation}
\label{killing-2}
\bar\nabla_{Y}Y = \frac{1}{2}\frac{\nabla \gamma}{\gamma^2}.
\end{equation}
what implies that
\begin{equation}
\langle \bar\nabla_Y Y, \nabla u\rangle = -\langle \bar\nabla_{\nabla u} Y, Y \rangle = \frac{1}{2\gamma^2}\langle\nabla \gamma, \nabla u\rangle.
\end{equation}
Using this one easily verifies that (\ref{Hnonpar}) may be written in divergence form as
\begin{equation}
\label{Hdiv}
\textrm{div} \frac{\nabla u}{W} -\frac{1}{2\gamma W}\langle \nabla \gamma, \nabla u\rangle = nH.
\end{equation}
In fact we have
\begin{eqnarray*}
& &\Big(\frac{u^i}{W}\Big)_{;i}  =\frac{1}{W}u^i_{;i} - \frac{1}{W^3}u^i u^j u_{i;j} -\frac{1}{2W^3}u^i \gamma_i.
\end{eqnarray*}
It is worth to point out that (\ref{Hdiv}) is equivalent to
\begin{equation}
\textrm{div} \frac{\nabla u}{W} -\frac{\gamma}{W}\langle \bar\nabla_Y Y, \nabla u\rangle = nH.
\end{equation}
We conclude that (\ref{derX}) may be written nonparametrically as
\begin{equation}
\label{uder}
\frac{\partial u}{\partial t} = W\textrm{div} \frac{\nabla u}{W} -W\mathcal{H}-\gamma\langle \bar\nabla_Y Y, \nabla u\rangle.
\end{equation}
Indeed it holds that
\begin{eqnarray*}
nH-\mathcal{H} = \langle \frac{\partial X}{\partial t}, N\rangle = \langle \frac{\partial u}{\partial t}\vartheta_*\frac{\partial }{\partial x^0}, \frac{\gamma}{W}\vartheta_*\frac{\partial}{\partial x^0}\rangle =\frac{1}{W} \frac{\partial u}{\partial t}.
\end{eqnarray*}
Using (\ref{Hnonpar}) one verifies that (\ref{uder}) is equivalent to
\begin{equation}
\label{eqtn-ut}
\frac{\partial u}{\partial t} = \Big(\sigma^{ij}-\frac{u^i}{W}\frac{u^j}{W}\Big)u_{i;j} -\frac{2\gamma+|\nabla u|^2}{W^2}\langle \frac{\bar\nabla \gamma}{2\gamma}, \nabla u\rangle
-W\mathcal{H}.
\end{equation}
We conclude that the Neumann problem (\ref{derX})-(\ref{neumann}) has the following nonparametric form
\begin{eqnarray}
\label{equation-nonpar}
& & u_t=\Big(\sigma^{ij}-\frac{u^i}{W}\frac{u^j}{W}\Big)u_{i;j}-\Big(\frac{1}{2\gamma} + \frac{1}{2W^2}\Big)\gamma^i u_i-W\mathcal{H} \quad{\rm in}\quad  \Omega\times[0,T)\\
\label{initial-nonpar}& & u(\cdot, 0)=u_0(\cdot) \quad\quad\quad\quad\quad\quad\quad\quad\quad\quad\quad\quad\quad\quad\quad\quad\quad\, {\rm in}\quad \Omega\times\{0\}
\end{eqnarray}
 with boundary condition
\begin{equation}
\label{neumann-nonpar}
\langle N,\nu\rangle =\phi \quad {\rm on} \quad  \partial\Omega\times[0,T).
\end{equation}
This boundary value problem
describes the evolution of the Killing graph of the function $u(\cdot,t)$ by its mean curvature in the direction of the unit normal $N$ with prescribed contact angle at the boundary.

The standard theory for quasilinear parabolic equations \cite{liebermann} guarantees that  the problem of solving (\ref{derX})-(\ref{neumann}) is reduced to obtaning \emph{a priori} height and gradient estimates for solutions to (\ref{equation-nonpar})-(\ref{neumann-nonpar}).

\section{Height estimates}\label{section-height}

From now on, we consider the parabolic linear operator given by
\begin{equation}
\label{linear}
\mathcal{L}v= g^{ij}v_{i;j} - \Big(\frac{1}{2\gamma}+\frac{1}{2W^2}\Big)\gamma^i v_i -\mathcal{H}\frac{u^i}{W}v_i - v_t,
\end{equation}
where $v\in C^{\infty}(\Omega\times[0,T))$.

\begin{proposition}\label{height-prop}
For a solution $u\in C^{\infty}(\bar\Omega\times[0,T^*])$, $T^*<T$, of {\rm (\ref{equation-nonpar})-(\ref{neumann-nonpar})}, it holds that
\begin{displaymath}
\max\limits_{\bar\Omega\times[0,T^*]}|u_t|=\max\limits_{\bar\Omega}|u_t(0,\cdot)|.
\end{displaymath}
Then it follows that
\begin{displaymath}
\max\limits_{\bar\Omega\times[0,T^*]}|u|\le CT^*
\end{displaymath}
for a given constant $C>0$  which depends on $T^*$.
\end{proposition}

\begin{proof} First of all we verify that $u_t$ is a solution for a linear parabolic equation. Indeed one has
\begin{eqnarray*}
\mathcal{L}u_t&=&g^{ij}u_{ti;j}-\Big(\frac{1}{2\gamma} + \frac{1}{2W^2}\Big)\langle\nabla\gamma,\nabla u_t\rangle-u_{tt}\nonumber\\
&=&(g^{ij}u_{i;j})_t-g^{ij}_{;t}u_{i;j}-\Big(\frac{1}{2\gamma} + \frac{1}{2W^2}\Big)\langle\nabla\gamma,\nabla u_t\rangle-u_{tt}\nonumber\\
&=&-g^{ij}_{;t}u_{i;j}+\Big(\frac{1}{2\gamma} + \frac{1}{2W^2}\Big)_t\langle\nabla\gamma,\nabla u\rangle+\Big(\frac{1}{2\gamma} + \frac{1}{2W^2}\Big)\langle\nabla\gamma_t,\nabla u\rangle
+W_t \mathcal{H}.
\end{eqnarray*}
However since $\gamma =\gamma(x)$ in (\ref{eqtn-ut}) and $x$ is independent of $t$ it follows that
\begin{eqnarray*}
& & \Big(\frac{1}{2\gamma} + \frac{1}{2W^2}\Big)_t = \Big(\frac{1}{2\gamma}\Big)_t -\frac{1}{W^3}W_t = -\frac{1}{2W^4}(\gamma_t+2u^k u_{k;t})=-\frac{1}{W^4}u^k u_{t;k}.
\end{eqnarray*}
In the same way we have
\begin{equation}
\label{Wt-nonpar}
W_t = \frac{1}{2W}(\gamma_t +2u^k u_{k;t})= \frac{1}{W}u^k u_{t;k}.
\end{equation}
We conclude that
\begin{equation*}
\mathcal{L}u_t = -g^{ij}_{;t}u_{i;j}-\frac{1}{W^4}\langle\nabla\gamma,\nabla u\rangle u^k (u_t)_{k}+\frac{1}{W}\mathcal{H}u^k (u_t)_k.
\end{equation*}
Now using the fact that $\sigma^{ij}_{;t}=0$ and $\gamma_t =0$  we have
\begin{eqnarray*}
\mathcal{L}u_t &=&\frac{2}{W}\Big(\frac{u^i_{;t}u^j}{W}-\frac{u^i}{W}\frac{u^j}{W}W_t\Big)u_{i;j}-\frac{1}{W^4}\langle\nabla\gamma,\nabla u\rangle u^k (u_t)_{k}+\frac{1}{W}\mathcal{H}u^k (u_t)_k\\
&=&\frac{2}{W}\Big((W_i-\frac{\gamma_i}{2W})u^i_{;t}-(W_i-\frac{\gamma_i}{2W})\frac{u^i}{W}\frac{u^k}{W} u_{t;k}\Big)-\frac{1}{W^4}\langle\nabla\gamma,\nabla u\rangle u^k (u_t)_{k}+\frac{1}{W}\mathcal{H}u^k (u_t)_k\\
&=&\frac{2}{W}(W_i-\frac{\gamma_i}{2W})(\sigma^{ik}-\frac{u^i}{W}\frac{u^k}{W})u_{t;k}-\frac{1}{W^4}\langle\nabla\gamma,\nabla u\rangle u^k (u_t)_{k}+\frac{1}{W}\mathcal{H}u^k (u_t)_k.
\end{eqnarray*}
Hence it follows that
\begin{equation}
\mathcal{L}u_t -\frac{2}{W}g^{ik}(W_i-\frac{\gamma_i}{2W})(u_{t})_k+\frac{1}{W^4}\langle\nabla\gamma,\nabla u\rangle u^k (u_t)_{k}-\frac{1}{W}\mathcal{H}u^k (u_t)_k= 0.
\end{equation}
Thus  fixed $T^*\in [0,T)$ let $(x_0, t_0)$ be  a  point in $\bar\Omega\times [0,T^*]$ such that
\[
u_t (x_0, t_0) = \max\limits_{\bar\Omega\times[0,T^*]}|u_t|.
\]
Hence we choose a coordinate system adapted to the boundary $\Gamma$ in such a way that $\frac{\partial}{\partial x^n}=\nu$ at $x_0$. Then, at the point $(x_0,t_0)$ we have
\[
u_{i;t}=u_{t;i}=0
\]
for $1\leq i<n$ what implies that
\begin{eqnarray*}
W_t=\frac{1}{W}u^nu_{n;t}=-\phi(x_0)u_{n;t},
\end{eqnarray*}
where we used (\ref{neumann-nonpar}) and (\ref{Wt-nonpar}).  On the other hand, (\ref{neumann-nonpar}) implies that
\begin{eqnarray}
u_{t;n}=u_{n;t}=-(\phi W)_t=-\phi (x_0)W_t.
\end{eqnarray}
at $(x_0, t_0)$. We conclude that
\[
(1-\phi^2(x_0))u_{n;t}=0.
\]
However since $\mid\phi\mid<1$, it follows that $u_{t;n}=0$ what contradicts the parabolic Hopf Lemma \cite{liebermann}.

From this contradiction we conclude that $t_0 = 0$. Since $T^*$ is arbitrary, the conclusion follows.
\end{proof}

\section{Boundary gradient estimates}\label{section-boundary}

Now we will prove a gradient bound for a solution of (\ref{equation-nonpar})-(\ref{neumann-nonpar}) by applying a modification of the Korevaar's technique \cite{korevaar} which appeared formerly in \cite{guan}.

From now on, we consider a non-negative extension $d:\bar\Omega\to \mathbb{R}$ of the distance function $\textrm{dist}_P(\cdot,\Gamma)$ satisfying $|\nabla d|\le 1$ in $\bar\Omega$. In the same way, we consider a $C^\infty$ extension of the boundary data $\phi$ to the domain $\bar\Omega$ which we denote also by $\phi$.  Then we define
\begin{equation}
\eta=e^{Ku} h
\end{equation}
where
\begin{equation}
h= 1+ \alpha d-\phi\langle\nabla d,N\rangle,
\end{equation}
where $K$ and $\alpha$ are positive numbers to be fixed later.

\begin{proposition}\label{boundary-prop}
For $\alpha>0$ sufficiently large independent of $K$ and $t$, if for some $t\geq0$ fixed, $\eta W(\cdot,t)$ attains a local maximum value at a point $x_0\in\partial\Omega$, then $W(x_0,t)\leq K$.
\end{proposition}

\begin{proof} Let $t\ge 0$ be such that
\begin{displaymath}
\max\limits_{\bar\Omega}\eta W(t,\cdot)  = \eta W(t,x_0)
\end{displaymath}
for a point $x_0\in\Gamma$. Hence we choose a coordinate system adapted to $\Gamma$ such that  $\frac{\partial}{\partial x^n}=\nu$ at $x_0$ and
\begin{equation}
u_1(x_0)\geq0 \quad \textrm{and}\quad u_i(x_0)=0, \quad  \textrm{for} \quad 2\leq i\leq n-1.
\end{equation}
We have at $x_0$
\begin{eqnarray}
0=(\eta W)_1=\eta_1W+\eta W_1=e^{Ku}\big(WKu_1(1-\phi^2)-2W\phi\phi_1+W_1(1-\phi^2)\big)
\end{eqnarray}
from what follows that
\begin{eqnarray}
W_1=-Ku_1W+\frac{2\phi\phi_1}{(1-\phi^2)}W.
\end{eqnarray}
On the other hand at $x_0$ we have
\begin{eqnarray}
\eta_n&=&e^{Ku}\big(Ku_n(1-\phi^2)+\alpha-\phi\phi_n-\phi(\langle\nabla_{\nabla d}N,\nabla d\rangle+\langle N,\nabla_{\nabla d}\nabla d\rangle)\big)\nonumber\\
&=&e^{Ku}\big(Ku_n(1-\phi^2)+\alpha-\phi\phi_n-\phi(\langle\partial_n\frac{1}{W}(\gamma Y-\nabla u),\partial_n\rangle+\langle\frac{1}{W}\nabla_{\partial_n}(\gamma Y-\nabla u),\partial_n\rangle)\big)\nonumber\\
&=&e^{Ku}\big(Ku_n(1-\phi^2)+\alpha-\phi\phi_n-\frac{1}{W^2}\phi u_nW_n+\frac{1}{W}\phi u_{n;n}\big).\nonumber
\end{eqnarray}
Since $(\eta W)_n\leq 0$ at $x_0$ it holds that
\begin{eqnarray}
0&\geq& WKu_n(1-\phi^2)+\alpha W-W\phi\phi_n-\frac{1}{W}\phi u_nW_n+\phi u_{n;n}+(1-\phi^2)W_n\nonumber\\
&=&WKu_n(1-\phi^2)+\alpha W+W_n+\phi u_{n;n}+u_n\phi_n\nonumber\\
&=&WKu_n(1-\phi^2)+\alpha W+W_n+\phi u_{n;n}-W\phi\phi_n.\nonumber
\end{eqnarray}
On the other hand
\begin{eqnarray}
& & W_n=\frac{\gamma_n}{2W}+\frac{1}{W}(u_1u_{1;n}+u_nu_{n;n})=\frac{\gamma_n}{2W}-\frac{1}{W}\phi u_1W_1-\phi_1u_1-\phi u_{n;n}
\end{eqnarray}
what implies that
\begin{eqnarray}
W_n&=&\frac{\gamma_n}{2W}-\frac{1}{W}\phi u_1\Big(\frac{2\phi\phi_1 W}{1-\phi^2}-Ku_1W\Big)-\phi_1u_1-\phi u_{n;n}\nonumber\\
&=&\frac{\gamma_n}{2W}-\frac{1+\phi^2}{1-\phi^2}u_1\phi_1+K\phi u_1^2-\phi u_{n;n}.\nonumber
\end{eqnarray}
Therefore since
\begin{eqnarray}
u_1^2=|\nabla u|^2-u_n^2=W^2-\gamma-\phi^2W^2=W^2(1-\phi^2)-\gamma\nonumber
\end{eqnarray}
we conclude that
\begin{eqnarray}
0&\geq& \alpha+\frac{\gamma_n}{2W^2}-\frac{1+\phi^2}{1-\phi^2}\frac{u_1}{W}\phi_1+\frac{K\phi u_1^2}{W}-\phi\phi_n+Ku_n(1-\phi^2)\nonumber\\
&=&\alpha+\frac{\gamma_n}{2W^2}+\frac{1+\phi^2}{1-\phi^2}N_1\phi_1+K\phi\Big(W(1-\phi^2)-\frac{\gamma}{W}\Big)-\phi\phi_n-K\phi W(1-\phi^2)\nonumber\\
&=&\alpha+\frac{\gamma_n}{2W^2}+\frac{1+\phi^2}{1-\phi^2}N_1\phi_1-\frac{K\phi \gamma}{W}-\phi\phi_n\nonumber\\
&\geq&\alpha+C-\frac{K\gamma}{W},\nonumber
\end{eqnarray}
for a given constant $C$ depending solely on $\gamma$ and $\phi$. It follows that $W(x_0,t)\leq K$ if $\alpha$ is chosen large enough  and independent of $K$ and $t$.
\end{proof}

\section{Interior gradient estimates}\label{section-interior}

In this section we deduce a global gradient bound using the techniques in \cite{calle} and \cite{guan}. However the more general context of warped product gives rise to a long list of additional terms which require a careful tracking along the calculations.

In the sequel, we consider the parabolic linear operator given by
\begin{equation}
\label{linear}
Lv= g^{ij}v_{i;j} - \Big(\frac{1}{2\gamma}+\frac{1}{2W^2}\Big)\gamma^i v_i - v_t,
\end{equation}
where $v\in C^{\infty}(\Omega\times[0,T))$.

\begin{proposition}\label{interior-prop} For fixed $T^*<T$ there exists $K>0$ sufficiently large so that if
\[
\eta W(x_0,t_0)=\max\limits_{\bar{\Omega}\times [0,T^*]}\eta W
\]
for some $(x_0,t_0)\in\bar{\Omega}\times [0,T^*]$, then $W(x_0,t_0)\leq C$, for some constant $C$.
\end{proposition}

\begin{proof} We can assume $x_0\in\Omega$ and $t_0>0$. At a point $(x_0,t_0)$ where $\eta W$ attains maximum value we have
\begin{eqnarray}
\eta_i W+\eta W_i=0
\end{eqnarray}
and
\begin{equation}
\frac{1}{\eta}L\eta+\frac{1}{W}\Big(LW-\frac{2}{W}g^{ij}W_iW_j\Big)\leq 0.
\end{equation}
We conclude that
\begin{eqnarray*}
\frac{1}{\eta}L\eta  &=& KLu + \frac{1}{h}Lh +K^2 g^{ij}u_iu_j+2Kg^{ij}u_i\frac{h_j}{h}\\
 &=& K\mathcal{H}W+ \frac{1}{h}Lh +K^2 \frac{\gamma|\nabla u|^2}{W^2}+2Kg^{ij}u_i\frac{h_j}{h}.
\end{eqnarray*}
Now we have
\begin{eqnarray*}
g^{ij}u_i h_j = \frac{\gamma}{W^2} u^j h_j = -\frac{\gamma}{W}(\alpha \langle N, \nabla d\rangle -\langle N, \nabla \phi\rangle \theta - \phi\langle N, \nabla \theta\rangle).
\end{eqnarray*}
However
\begin{eqnarray*}
 \langle N, \nabla \theta \rangle = W\langle AY^T, \nabla^\Sigma d\rangle +\langle \nabla_{\frac{\nabla u}{W}} \nabla d, \frac{\nabla u}{W}\rangle -\kappa\frac{|\nabla u|^2}{W^2}.
\end{eqnarray*}
Therefore
\begin{eqnarray*}
g^{ij}u_i h_j = -\alpha\frac{\gamma}{W}\langle N, \nabla d\rangle +\frac{\gamma}{W}\langle N, \nabla \phi\rangle \theta +\gamma \phi \langle AY^T, \nabla^\Sigma d\rangle +\frac{\gamma}{W} \phi\langle \nabla_{\frac{\nabla u}{W}} \nabla d, \frac{\nabla u}{W}\rangle -\gamma \phi\kappa\frac{|\nabla u|^2}{W^3}.
\end{eqnarray*}
Thus the expression for $Lh$ in Appendix allows us to conclude that
\begin{eqnarray*}
& & \frac{1}{\eta}L\eta = K\mathcal{H}W + K^2 \frac{\gamma|\nabla u|^2}{W^2} \\
& & \,\,\,\,+\frac{2K}{h}\big(-\alpha\frac{\gamma}{W}\langle N, \nabla d\rangle +\frac{\gamma}{W}\langle N, \nabla \phi\rangle \theta +\gamma \phi \langle AY^T, \nabla^\Sigma d\rangle +\frac{\gamma}{W} \phi\langle \nabla_{\frac{\nabla u}{W}} \nabla d, \frac{\nabla u}{W}\rangle -\gamma \phi\kappa\frac{|\nabla u|^2}{W^3}\big)\\
& & \,\,\,\,  +\frac{1}{h}|A|^2 \phi\theta + n\frac{1}{h}\phi HW\langle AY^T, \nabla^\Sigma d\rangle
+\frac{1}{h}\big(\kappa\gamma-\frac{1}{2}\langle \nabla d, \nabla \gamma\rangle\big)\phi\langle AY^T, Y^T\rangle\\
& &\,\,\,\, +\frac{2}{h}\langle A\nabla^\Sigma d, \nabla^\Sigma\phi\rangle +\frac{2}{h}\langle A, \nabla^2 d\rangle_\Sigma\phi-\frac{1}{hW^2}\phi\langle A\nabla^\Sigma d, X_*\nabla\gamma\rangle\\
& & \,\,\,\,\ + \frac{1}{h}\big(nH-\mathcal{H}\big)\big(\langle N, \nabla \phi\rangle \theta-\alpha\theta-\frac{1}{2W^2}\langle \nabla\gamma, \nabla d\rangle\phi+\langle \nabla_{\frac{\nabla u}{W}}\nabla d, \frac{\nabla u}{W}\rangle\phi\big)-n\frac{1}{h}\kappa H\phi\\
& & \,\,\,\,-n\frac{\alpha}{h} H_d +\frac{1}{h}\big(2\langle N, \nabla\phi\rangle -\alpha\big)\langle \nabla_{\frac{\nabla u}{W}}\nabla d, \frac{\nabla u}{W}\rangle +\frac{2}{h}\langle \nabla_{\frac{\nabla u}{W}}\nabla d, \nabla \phi\rangle
\\
& &\,\,\,\,-\frac{1}{h}\phi \langle \nabla_{\frac{\nabla u}{W}}\nabla d, \frac{\nabla \gamma}{2\gamma}\rangle +\frac{1}{h}\phi\langle \nabla^\Sigma \mathcal{H}, \bar\nabla d\rangle+n\frac{1}{h}\langle \nabla H_d, N\rangle \phi-\frac{1}{h}\phi\nabla^3 d(\frac{\nabla u}{W},\frac{\nabla u}{W},\frac{\nabla u}{W}) \\
& & \,\,\,\,+\frac{1}{h}\textrm{Ric}(\nabla d, \frac{\nabla u}{W})\phi+\frac{\gamma}{hW^2}\langle N, \nabla \kappa\rangle\phi
-\frac{1}{h}\big(\frac{1}{2\gamma}+\frac{1}{2W^2}\big)\alpha\langle\nabla d, \nabla \gamma\rangle  +\frac{1}{h}\big(\frac{1}{2\gamma}+\frac{1}{2W^2}\big)\langle \nabla\phi,\nabla \gamma\rangle \theta \\
& & \,\,\,\, -\kappa\frac{1}{h}\phi \langle N, \frac{\nabla \gamma}{2\gamma}\rangle+\frac{2}{h}\kappa\frac{\gamma}{W^2}\langle \nabla\phi, N\rangle-\frac{1}{h}\big(\Delta \phi -\langle \nabla_{\frac{\nabla u}{W}}\nabla \phi, \frac{\nabla u}{W}\rangle\big)\theta.
\end{eqnarray*}
On the other hand Lemma \ref{LW-lemma} yields
\begin{eqnarray*}
& & \frac{1}{W}\Big(LW-\frac{2}{W}g^{ij}W_i W_j\Big) = |A|^2 +nHW^2\langle AY^T,  Y^T\rangle -nHW^2\langle \frac{\nabla \gamma}{2\gamma^2}, N\rangle\\
& & \,\,\,\, -3\frac{1}{W}\gamma\langle A Y^T, X_*\frac{\nabla\gamma}{2\gamma}\rangle  + g^{ij}\frac{\gamma_{i;j}}{2\gamma} - \frac{3}{4}\frac{|\nabla\gamma|^2}{4\gamma^2}  -\frac{1}{4}\langle \frac{\nabla \gamma}{2\gamma}, N\rangle^2 + \gamma \langle \bar\nabla_{N}\frac{\bar\nabla \gamma}{2\gamma^2}, N\rangle\\
& &\,\,\,\, - \langle \nabla^\Sigma \mathcal{H}, N\rangle - \frac{|\nabla \gamma|^2}{4\gamma}\frac{1}{W^2}-\frac{W_t}{W}.
\end{eqnarray*}
Now we use the fact that $x_0$ is a critical point to $\eta W$. We have
\[
e^{Ku}(Ku_i h + h_i)W=-e^{Ku}hW_i.
\]
what implies that
\[
-KW^2 h N_i N^i+ Wh_i N^i = -hW_i N^i
\]
and then
\[
-Kh |\nabla u|^2+ Wh_i N^i = -hW_i N^i.
\]
However
\begin{eqnarray*}
& & W_i N^i =\frac{\gamma_i}{2\gamma}N^i W+N_i N^iW^3\langle \frac{\nabla \gamma}{2\gamma^2}, N\rangle +W^2\langle A Y^T, N^iX_*\frac{\partial}{\partial x^i}\rangle\\
& & \,\,=\frac{1}{2\gamma}\langle \nabla \gamma, N\rangle W+|\nabla u|^2 W\langle \frac{\nabla \gamma}{2\gamma^2}, N\rangle -W^3\langle A Y^T, Y^T\rangle
\end{eqnarray*}
and
\begin{eqnarray*}
& & h_i N^i = \alpha \theta -\langle \nabla\phi, N\rangle  \theta +\phi a_i^j N^i d_j -\phi (d_{i;j}N^i N^j-\kappa \sigma_{ij})N^iN^j\\
& & \,\,= \alpha \theta -\langle \nabla\phi, N\rangle  \theta -\phi W\langle AY^T, \nabla^\Sigma d\rangle -\phi\langle \nabla_{\frac{\nabla u}{W}}\nabla d, \frac{\nabla u}{W}\rangle +\phi\kappa \frac{|\nabla u|^2}{W^2}.
\end{eqnarray*}
We then conclude that
\begin{eqnarray*}
& &-K \frac{|\nabla u|^2}{W} +\frac{\alpha \theta}{h} -\frac{1}{h}\langle \nabla\phi, N\rangle  \theta -\frac{\phi}{h} W\langle AY^T, \nabla^\Sigma d\rangle -\frac{\phi}{h}\langle \nabla_{\frac{\nabla u}{W}}\nabla d, \frac{\nabla u}{W}\rangle +\frac{\phi}{h}\kappa \frac{|\nabla u|^2}{W^2}\\
& &\,\, =-\frac{1}{2\gamma}\langle \nabla \gamma, N\rangle -|\nabla u|^2 \langle \frac{\nabla \gamma}{2\gamma^2}, N\rangle +W^2\langle A Y^T, Y^T\rangle
\end{eqnarray*}
Moreover
\begin{eqnarray*}
 -\frac{W_t}{W} &=&\frac{\eta_t}{\eta}=Ku_t+\frac{h_t}{h}=WK(nH-\mathcal{H})+\frac{h_t}{h}\nonumber\\
&=&nHKW-KW\mathcal{H}-\frac{1}{h}(nH-\mathcal{H})\Big(\langle\nabla\phi,N\rangle\theta - \alpha\theta  - \frac{\phi}{2W^2}\langle\nabla\gamma,\nabla d\rangle\Big) \nonumber\\
& &\,\,\,\, +\frac{\phi}{h}\langle\frac{\nabla u}{W},\nabla_{\frac{\nabla u}{W}}\nabla d\rangle.\nonumber
\end{eqnarray*}
Then we have
\begin{eqnarray*}
& & \frac{1}{W}\Big(LW-\frac{2}{W}g^{ij}W_i W_j\Big) = |A|^2 +KnH \frac{\gamma}{W} +\frac{\alpha \theta}{h}nH -\frac{1}{h}nH\langle \nabla\phi, N\rangle  \theta  -\frac{\phi}{h} nHW\langle AY^T, \nabla^\Sigma d\rangle \nonumber\\
& &- \frac{1}{h}(nH-\mathcal{H})\Big(\langle\nabla\phi,N\rangle\theta - \alpha\theta  - \frac{\phi}{2W^2}\langle\nabla\gamma,\nabla d\rangle\Big)-\frac{\phi}{h}nH\langle \nabla_{\frac{\nabla u}{W}}\nabla d, \frac{\nabla u}{W}\rangle +\frac{\phi}{h}nH\kappa \frac{|\nabla u|^2}{W^2} \\
& &-3\frac{1}{W}\gamma\langle A Y^T, X_*\frac{\nabla\gamma}{2\gamma}\rangle
+ g^{ij}\frac{\gamma_{i;j}}{2\gamma} - \frac{3}{4}\frac{|\nabla\gamma|^2}{4\gamma^2}  -\frac{1}{4}\langle \frac{\nabla \gamma}{2\gamma}, N\rangle^2 + \gamma \langle \bar\nabla_{N}\frac{\bar\nabla \gamma}{2\gamma^2}, N\rangle - \langle \nabla^\Sigma \mathcal{H}, N\rangle \\
& &- \frac{|\nabla \gamma|^2}{4\gamma}\frac{1}{W^2}-KW\mathcal{H}  +\frac{\phi}{h}\langle\frac{\nabla u}{W},\nabla_{\frac{\nabla u}{W}}\nabla d\rangle.\nonumber
\end{eqnarray*}
We conclude that
\begin{eqnarray*}
& & \frac{1}{\eta}L\eta +\frac{1}{W}\Big(LW-\frac{2}{W}g^{ij}W_i W_j\Big) =  K^2 \frac{\gamma|\nabla u|^2}{W^2} +\mathcal{A}+\mathcal{B},
\end{eqnarray*}
where
\begin{eqnarray*}
& & \mathcal{A} = \Big(1+\frac{\phi\theta}{h}\Big)|A|^2+\frac{2K}{h}\gamma \phi \langle AY^T, \nabla^\Sigma d\rangle+\frac{\phi}{h}\big(\kappa\gamma-\frac{1}{2}\langle \nabla d, \nabla \gamma\rangle\big)\langle AY^T, Y^T\rangle\\
& & \,\,\,\,
 +\frac{2}{h}\langle A\nabla^\Sigma d, \nabla^\Sigma\phi\rangle +\frac{2}{h}\langle A, \nabla^2 d\rangle_\Sigma\phi-\frac{1}{hW^2}\phi\langle A\nabla^\Sigma d, X_*\nabla\gamma\rangle\\
& &  \,\,\,\,+KnH \frac{\gamma}{W} +\frac{\alpha \theta}{h}nH -\frac{1}{h}nH\langle \nabla\phi, N\rangle  \theta -\frac{\phi}{h}nH\kappa \frac{\gamma}{W^2} -3\frac{1}{W}\gamma\langle A Y^T, X_*\frac{\nabla\gamma}{2\gamma}\rangle
\end{eqnarray*}
and
\begin{eqnarray*}
& & \mathcal{B}=\frac{2K}{h}\big(-\alpha\frac{\gamma}{W}\langle N, \nabla d\rangle +\frac{\gamma}{W}\langle N, \nabla \phi\rangle \theta  +\frac{\gamma}{W} \phi\langle \nabla_{\frac{\nabla u}{W}} \nabla d, \frac{\nabla u}{W}\rangle -\gamma \phi\kappa\frac{|\nabla u|^2}{W^3}\big) \\
& & \,\,-\mathcal{H}\langle \nabla_{\frac{\nabla u}{W}}\nabla d, \frac{\nabla u}{W}\rangle\phi-n\frac{\alpha}{h} H_d +\frac{1}{h}\big(2\langle N, \nabla\phi\rangle -\alpha\big)\langle \nabla_{\frac{\nabla u}{W}}\nabla d, \frac{\nabla u}{W}\rangle +\frac{2}{h}\langle \nabla_{\frac{\nabla u}{W}}\nabla d, \nabla \phi\rangle
\\
& &\,\,-\frac{1}{h}\phi \langle \nabla_{\frac{\nabla u}{W}}\nabla d, \frac{\nabla \gamma}{2\gamma}\rangle +\frac{1}{h}\phi\langle \nabla^\Sigma \mathcal{H}, \bar\nabla d\rangle+n\frac{1}{h}\langle \nabla H_d, N\rangle \phi-\frac{1}{h}\phi\nabla^3 d(\frac{\nabla u}{W},\frac{\nabla u}{W},\frac{\nabla u}{W}) \\
& & \,\,+\frac{1}{h}\textrm{Ric}(\nabla d, \frac{\nabla u}{W})\phi+\frac{\gamma}{hW^2}\langle N, \nabla \kappa\rangle\phi
-\frac{1}{h}\big(\frac{1}{2\gamma}+\frac{1}{2W^2}\big)\alpha\langle\nabla d, \nabla \gamma\rangle   \\
& & \,\, +\frac{1}{h}\big(\frac{1}{2\gamma}+\frac{1}{2W^2}\big)\langle \nabla\phi,\nabla \gamma\rangle \theta-\kappa\frac{1}{h}\phi \langle N, \frac{\nabla \gamma}{2\gamma}\rangle+\frac{2}{h}\kappa\frac{\gamma}{W^2}\langle \nabla\phi, N\rangle \\
& &\,\,-\frac{1}{h}\big(\Delta \phi -\langle \nabla_{\frac{\nabla u}{W}}\nabla \phi, \frac{\nabla u}{W}\rangle\big)\theta+ g^{ij}\frac{\gamma_{i;j}}{2\gamma} - \frac{3}{4}\frac{|\nabla\gamma|^2}{4\gamma^2}  -\frac{1}{4}\langle \frac{\nabla \gamma}{2\gamma}, N\rangle^2 + \gamma \langle \bar\nabla_{N}\frac{\bar\nabla \gamma}{2\gamma^2}, N\rangle \\
& &\,\, - \langle \nabla^\Sigma \mathcal{H}, N\rangle- \frac{|\nabla \gamma|^2}{4\gamma}\frac{1}{W^2}  +\frac{\phi}{h}\langle\frac{\nabla u}{W},\nabla_{\frac{\nabla u}{W}}\nabla d\rangle.\nonumber
\end{eqnarray*}
However using some standard inequalites we obtain
\begin{eqnarray*}
& & \mathcal{A} \ge  \Big(1+\frac{\phi\theta}{h}\Big)|A|^2-\Big(\frac{2K\gamma}{h\sqrt{\gamma}}+\frac{\kappa}{h}+\frac{1}{2h\gamma}
|\nabla\gamma|+\frac{2}{h}|\nabla\phi|+\frac{2}{h}|\nabla^2d|_{\Sigma}\\
& &+\frac{1}{hW^2}|X_*\nabla\gamma|+\frac{K\gamma\sqrt{n}}{W}+\frac{\alpha\theta\sqrt{n}}{h}+\frac{\theta\sqrt{n}}{h}|\nabla\phi|+
\frac{\gamma\sqrt{n}\kappa}{hW^2}+\frac{3\gamma}{\sqrt{\gamma}W}|X_*\frac{\nabla\gamma}{2\gamma}|\Big)|A|\\
\end{eqnarray*}
Using that $W^2\ge \gamma$ and choosing $\alpha$ sufficiently large and depending only on $n, \gamma, \phi$ and $\kappa$  we have
\begin{eqnarray*}
& & \mathcal{A}  \geq \frac{1}{2}|A|^2-\Big(\epsilon+2\sqrt{\gamma}\frac{K}{h}+\frac{K\gamma\sqrt{n}}{W}+\frac{3\sqrt{\gamma}}{W}
|X_*\frac{\nabla\gamma}{2\gamma}|\Big)|A|\\
& & \,\,\geq-\Big(\epsilon+2\sqrt{\gamma}\frac{K}{h}+\frac{K\gamma\sqrt{n}}{W}+\frac{3\sqrt{\gamma}}{W}
|X_*\frac{\nabla\gamma}{2\gamma}|\Big)^2.
\end{eqnarray*}
Moreover
\begin{eqnarray*}
\mathcal{B} \geq -C \bigg(1+\frac{\alpha}{h}+\frac{\alpha}{hW^2}+\frac{1}{h}+\frac{1}{W^2}+\frac{1}{hW^2}+K\frac{\alpha}{h}+\frac{K}{h}\bigg),
\end{eqnarray*}
where $C$ is a constant depending on $n, \gamma, \phi, d, \kappa$ and $\mathcal{H}$.

Hence we obtain
\begin{eqnarray*}
& & \frac{1}{\eta}L\eta +\frac{1}{W}\Big(\mathcal{L}W-\frac{2}{W}g^{ij}W_i W_j\Big) \geq   K^2 \frac{\gamma|\nabla u|^2}{W^2}-C(\epsilon)-\frac{K}{W}C(\epsilon,\gamma,n)-\frac{K^2}{W^2}C(\gamma,n)\nonumber\\
& & \,\,\,\, -\frac{1}{W}C(\gamma,\epsilon)-\frac{K}{W^2}C(\gamma,n)-\frac{K^2}{h^2}C(\gamma)-\frac{K^2}{hW}C(\gamma, n)-\frac{1}{W^2}C(\gamma)-\frac{K}{hW}C(\gamma)
\\
& &\,\,\,\,-K\frac{\alpha}{h}C-\frac{K}{h}C(\epsilon, \gamma)-C-\frac{\alpha}{h}C-\frac{\alpha}{hW^2}C -\frac{1}{h}C-\frac{1}{W^2}C-\frac{1}{hW^2}C.
\end{eqnarray*}
Then
\begin{eqnarray}
& & - K^2 \frac{\gamma|\nabla u|^2}{W^2} \geq -C\bigg(\frac{K^2}{W^2}+\frac{K}{W^2}+\frac{1}{W}+\frac{K}{hW}+\frac{K^2}{h^2}+\frac{\alpha}{hW^2}+\frac{1}{W^2}+\frac{1}{hW^2}+\frac{K}{W}+\frac{1}{W}\nonumber\\
& & \,\,\,\, + K\frac{\alpha}{h}+\frac{K}{h}+\frac{\alpha}{h}+\frac{1}{h}+1\bigg).\nonumber
\end{eqnarray}
It follows that
\begin{eqnarray*}
& & \bigg(K^2\gamma-
\Big(\frac{K^2}{h^2}+K\frac{\alpha}{h}+\frac{K}{h}+\frac{1+\alpha}{h}+1\Big)C\bigg)W^2\leq \Big(K^2+K+\frac{1+\alpha}{h}+1\Big)C\\
& & \,\,\,\,+(K+\frac{K}{h}+1)CW.\nonumber
\end{eqnarray*}
Now suppose that $W(x_0,t_0)\ge 1$. Otherwise we are done. In this case we have $W \leq W^2$ and absorbing the terms with $W$ into that one with $W^2$ transforms the inequality above into
\begin{eqnarray*}
\Big(K^2\gamma-\frac{K^2}{h^2}C-\frac{K}{h}C-KC-C-
\frac{1}{h}(\alpha+1)(K+1)C\Big)W^2\leq \Big(K^2+K+1+\frac{1}{h}(\alpha+1)\Big)C.
\end{eqnarray*}
If $d_0 = d(x_0)$ then choosing $\alpha \ge 1/(C(d_0)d_0-1)$ for some constant $C(d_0)>1/d_0$ we obtain
$(1+\alpha)/h\le C(d_0)$ what implies that
\begin{eqnarray}
& & \Big(K^2\gamma-\frac{K^2}{h^2}C-\frac{K}{h}C-KC(d_0)-
C(d_0)\Big)W^2\leq (K^2+K+C(d_0))C.\nonumber
\end{eqnarray}
Then for $\alpha >\frac{1}{d_0}\max\{1,\sqrt{2C/\gamma}\}$ we have
\begin{eqnarray}
& & \Big(K^2\frac{\gamma}{2}-KC(d_0)-
C(d_0)\Big)W^2\leq (K^2+K+C(d_0))C.\nonumber
\end{eqnarray}
It follows that for $K > \frac{C(d_0)+\sqrt{C(d_0)^2+2\gamma C(d_0)}}{\gamma}$ we have $K^2\frac{\gamma}{2}-KC(d_0)-
C(d_0)>0$ and
\begin{equation}
W^2\leq \frac{C(K^2+K+C(d_0))}{K^2\frac{\gamma}{2}-KC(d_0)-C(d_0)}.
\end{equation}
This finishes the proof of the proposition.
\end{proof}

\begin{theorem}
There exists a unique solution $u:\bar\Omega\times [0,\infty)\to\mathbb{I}$ to the problem {\rm (\ref{derX})-(\ref{neumann})}.
\end{theorem}

\begin{proof}
Propositions \ref{height-prop}, \ref{boundary-prop} and \ref{interior-prop} yield the following global gradient bound
\begin{equation}
W(x,t) \le W(x_0, t_0) \frac{\eta(x,t)}{\eta(x_0, t_0)} \le C_1 e^{C_2 MT^*},
\end{equation}
for $(x,t)\in \bar\Omega \times [0,T^*]$, where $C_1$ and $C_2$ are positive constants and
\[
M =\max\limits_{\bar\Omega\times [0,T^*]} |u-u_0|.
\]
It results that (\ref{equation-nonpar}) is uniformly parabolic and then the standard theory of quasilinear parabolic PDEs may be applied for assuring the existence of a unique smooth solution to (\ref{equation-nonpar})-(\ref{neumann-nonpar}).
\end{proof}

\section{Asymptotic behavior}\label{section-asymptotic}

In the particular case when $u(x,t)=v(x)+Ct, \, (x,t)\in \bar\Omega\times [0,T),$  the initial value problem (\ref{equation-nonpar})-(\ref{neumann-nonpar}) becomes
\begin{eqnarray}
\label{elliptic-equation}
& & \textrm{div}\frac{\nabla v}{W}-\gamma\langle \bar\nabla_Y Y, \frac{\nabla v}{W} \rangle=\mathcal{H}+\frac{C}{W}\quad \textrm{in}\quad \Omega\\
\label{neumann-elliptic}& & \langle \nu,N\rangle=\phi(x,v) \quad\quad \quad\quad \quad\quad \quad\quad\,\,\textrm{on}\quad \partial\Omega
\end{eqnarray}
Conversely, notice that if $v(x)$ is a solution of (\ref{elliptic-equation})-(\ref{neumann-elliptic}) then $u=v+Ct$ is a solution of (\ref{equation-nonpar}) which is translating along the flow lines of $Y$ with speed $C$.

Now observe that
\[
\textrm{div}\frac{\nabla v}{W}-\gamma\langle \bar\nabla_Y Y, \frac{\nabla v}{W} \rangle = \textrm{div}\frac{\nabla v}{W}+\gamma\langle \bar\nabla_{\frac{\nabla v}{W}} Y, Y  \rangle=
\textrm{div}\frac{\nabla v}{W}+\gamma\langle \bar\nabla_Y \frac{\nabla v}{W} Y, Y  \rangle = \textrm{div}_M \frac{\nabla v}{W}.
\]
Therefore it follows from divergence theorem that
\begin{equation}
\int_{\vartheta([0,s]\times\bar\Omega)} \frac{C}{W}+\mathcal{H}=-\int_{\vartheta([0,s]\times\Gamma)} \langle\frac{\nabla v}{W}, \nu\rangle = \int_{\vartheta([0,s]\times\Gamma)} \langle N, \nu\rangle =\int_{\vartheta([0,s]\times\Gamma)} \phi .
\end{equation}
Since the integrands do not depend on $s$ we have
\begin{equation}
\int_{\Omega} C\frac{\sqrt{\gamma}}{W}+\sqrt{\gamma}\mathcal{H}=\int_{\Gamma} \phi\sqrt{\gamma}.
\end{equation}
from what results that
\begin{equation}
C=\frac{\int_{\Gamma}\phi\sqrt{\gamma} -\int_\Omega\mathcal{H}\sqrt{\gamma}}{\int_{\Omega}\frac{\sqrt{\gamma}}{W}}.
\end{equation}
Since $W\le C_1  e^{C_2 MT}$ and $|\phi|<1$ we conclude that
\begin{equation}
C\le \frac{|\Gamma| +\sup_\Omega |\mathcal{H}||\Omega|}{|\Omega|}C_1 e^{C_2MT}.
\end{equation}

Comparing an arbitrary solution of the mean curvature flow with translating graphs yields

\begin{proposition}
\label{height-M} Suppose that there exists a solution of {\rm (\ref{elliptic-equation})} for a given $C$. Then given a solution $u(x,t)$ of (\ref{equation-nonpar}) there exists a constant $M$ such that
\begin{equation}
\label{heightM}
|u(x,t)-Ct|\leq M
\end{equation}
for $(x,t)\in\bar\Omega\times [0,+\infty)$.
\end{proposition}

\begin{proof} Consider a solution $v$ of (\ref{elliptic-equation}). Then consider the functions $v_1=v+\inf\limits_\Omega (u_0-v)$ and $v_2 = v+\sup\limits_\Omega (u_0-v)$ which are also solutions of (\ref{elliptic-equation}). By definition we have $v_1\le u_0\le v_2$. Hence the parabolic maximum principle implies that
\[
v_1 + Ct \le u(\cdot, t) \le v_2 +Ct,
\]
for $t\in[0,T)$ from what we obtain (\ref{heightM}).
\end{proof}

\begin{theorem}
Suppose that there exists a solution of {\rm (\ref{elliptic-equation})} for  $C=0$.  Then $\lim_{t\rightarrow\infty}u_t=0$.  In particular the mean curvature flow converges to a graph with prescribed mean curvature $\mathcal{H}$ and prescribed contact angle $\phi$.
\end{theorem}

\begin{proof} Since $C=0$ Proposition (\ref{height-M}) implies that
\begin{equation}
\Big|\int_{\Gamma}u\phi\Big| \leq M\int_{\Gamma}|\phi|\leq M|\Gamma|,
\end{equation}
for $t\in [0,T)$.
Also we have
\begin{eqnarray*}
\frac{d}{dt}\int_{\Omega}W=\int_{\Omega}\frac{u^iu_{i;t}}{W}=-\int_{\Omega}\frac{u_t^2}{W}-\int_{\Omega}\frac{1}{2W^3}\langle\nabla u,\nabla\gamma\rangle -\int_{\Omega}\frac{|\nabla u|^2}{2\gamma W^2}\langle\nabla u,\nabla\gamma\rangle -\int_{\partial\Omega}u_t\phi.
\end{eqnarray*}
Therefore
\begin{eqnarray}
-\int_{\Omega}\frac{u_t^2}{W}=\frac{d}{dt}\bigg(\int_{\Omega}W+\int_{\partial\Omega}u\phi \bigg)+\int_{\Omega}\frac{1}{2W^3}\langle\nabla u,\nabla\gamma\rangle+\int_{\Omega}\frac{|\nabla u|^2}{2\gamma W^2}\langle\nabla u,\nabla\gamma\rangle.
\end{eqnarray}
It follows that
\begin{eqnarray}
& & \int_{0}^{T}\int_{\Omega}\frac{u_t^2}{W}=-\int_{\Omega}W(x,T)+\int_{\Omega}W(x,0)-\int_{\partial\Omega}u(x,T)\phi \nonumber\\
& &+\int_{\partial\Omega}u(x,0)\phi+\int_{0}^{T}\int_{\Omega}\frac{1}{2W^3}\langle\nabla u,\nabla\gamma\rangle +\int_{0}^{T}\int_{\Omega}\frac{|\nabla u|^2}{2\gamma W^2}\langle\nabla u,\nabla\gamma\rangle \leq C\nonumber
\end{eqnarray}
from what follows that $\lim\limits_{t\to\infty} \frac{u_t^2}{W}=0$. Since $W$ is bounded then $\lim\limits_{t\to\infty} u_t =0$. This finishes the proof of the theorem.
\end{proof}

\section{Appendix}

In what follows, $II$ and $A$ denote respectively the second fundamental form and the Weingarten map of $\Sigma_t$. Their components are given by
\begin{equation}
a_{ij} =II(X_*\frac{\partial}{\partial x^i},X_*\frac{\partial}{\partial x^j}):= \langle AX_*\frac{\partial}{\partial x^i},X_*\frac{\partial}{\partial x^j}\rangle
\end{equation}

Some lemmata will be needed in the sequel. Their content could be also of independent interest for other applications.

\begin{lemma} Denote $\theta=\langle\nabla d, N\rangle$. The differentials of the functions $\theta$ and $h$ have components given by
\begin{equation}
\theta_i =-a^j_i d_ j +(d_{i;j}-\kappa\sigma_{ij})N^j
\end{equation}
and
\begin{equation}
h_i = (\alpha\delta^j_i +\phi a^j_i) d_j -(\phi(d_{i;j}-\kappa\sigma_{ij})+\phi_i d_j)N^j
\end{equation}
respectively, where $\kappa=\langle \gamma \bar\nabla_Y Y , \nabla d\rangle$.
\end{lemma}

\begin{proof}  We have
\begin{eqnarray*}
& & \frac{\partial\theta}{\partial x^i} = X_*\frac{\partial}{\partial x^i}\langle N, \bar\nabla d\rangle = \langle \bar\nabla_{X_*\frac{\partial}{\partial x^i}}N, \bar\nabla d\rangle +\langle N, \bar\nabla_{X_*\frac{\partial}{\partial x^i}}\bar\nabla d\rangle \\
& &\,\,= -\langle AX_*\frac{\partial}{\partial x^i}, \bar\nabla d\rangle +\langle N, \bar\nabla_{\frac{\partial}{\partial x^i}+u_i \frac{\partial}{\partial x^0}}\bar\nabla d\rangle\\
& & \,\,= -\langle AX_*\frac{\partial}{\partial x^i}, \bar\nabla d\rangle +\frac{\gamma}{W}\langle \frac{\partial}{\partial x^0}, \bar\nabla_{\frac{\partial}{\partial x^i}}\bar\nabla d\rangle
-\langle \frac{\nabla u}{W}, \bar\nabla_{\frac{\partial}{\partial x^i}}\bar\nabla d\rangle\\
& & \,\,\,\,+u_i\frac{\gamma}{W}\langle \frac{\partial}{\partial x^0}, \bar\nabla_{ \frac{\partial}{\partial x^0}}\bar\nabla d\rangle
-u_i\langle \frac{\nabla u}{W}, \bar\nabla_{ \frac{\partial}{\partial x^0}}\bar\nabla d\rangle
\end{eqnarray*}
Since $P$ is totally geodesic  we have
\[
\langle \frac{\partial}{\partial x^0}, \bar\nabla_{\frac{\partial}{\partial x^i}}\bar\nabla d\rangle =\langle \frac{\partial}{\partial x^0}, \bar\nabla_{\frac{\partial}{\partial x^i}}\nabla d\rangle=0.
\]
Moreover we compute
\begin{eqnarray*}
\langle \frac{\partial}{\partial x^0}, \bar\nabla_{ \frac{\partial}{\partial x^0}}\bar\nabla d\rangle= |Y|^2\langle \frac{Y}{|Y|}, \bar\nabla_{ \frac{Y}{|Y|}}\bar\nabla d\rangle=|Y|^2 \kappa = \frac{1}{\gamma}\kappa
\end{eqnarray*}
and
\begin{eqnarray*}
& & \langle \frac{\nabla u}{W}, \bar\nabla_{ \frac{\partial}{\partial x^0}}\bar\nabla d\rangle = \langle \frac{\nabla u}{W}, \bar\nabla_{\bar\nabla d}  \frac{\partial}{\partial x^0}\rangle+\langle \frac{\nabla u}{W}, [\frac{\partial}{\partial x^0}, \bar\nabla d]\rangle=0,
\end{eqnarray*}
where we used the fact that $[\frac{\partial}{\partial x^0}, \bar\nabla d]=0$ and that $P$ is totally geodesic.

Thus we conclude that
\begin{eqnarray*}
& & \frac{\partial\theta}{\partial x^i} = -\langle AX_*\frac{\partial}{\partial x^i}, \bar\nabla d\rangle
-\langle \frac{\nabla u}{W}, \nabla_{\frac{\partial}{\partial x^i}}\nabla d\rangle+\kappa \frac{u_i}{W}.
\end{eqnarray*}
However
\begin{eqnarray*}
& & \langle AX_*\frac{\partial}{\partial x^i}, \bar\nabla d\rangle =a_i^j \langle X_*{\frac{\partial}{\partial x^j}},\bar\nabla d\rangle =
a_i^j \langle \frac{\partial}{\partial x^j}+u_j Y,\bar\nabla d\rangle = a_i^j d_j = g^{jk}a_{ik}d_j
\end{eqnarray*}
Therefore we write
\begin{equation}
\label{nabla-theta}
\theta_i = -g^{jk} a_{ik} d_j +(d_{i;j}-\kappa\sigma_{ij})N^j.
\end{equation}
This finishes the proof of the proposition.
\end{proof}

We denote the components of the tensor $X^*II$ in $P$ by
\begin{equation}
b_{ij} =X^*II(\frac{\partial}{\partial x^i},\frac{\partial}{\partial x^j}):= \langle AX_*\frac{\partial}{\partial x^i},X_*\frac{\partial}{\partial x^j}\rangle
\end{equation}
Notice that the covariant derivatives of $X^*II$ and $II$ are related by
\begin{eqnarray*}
& & \nabla_k b_{ij} =\langle(\nabla^\Sigma_{X_*\frac{\partial}{\partial x^k}} A)X_*\frac{\partial}{\partial x^i},X_*\frac{\partial}{\partial x^j}\rangle +\langle AX_*\frac{\partial}{\partial x^j}, \bar\nabla_{X_*\frac{\partial}{\partial x^k}}X_*\frac{\partial}{\partial x^i}-X_*\nabla_{\frac{\partial}{\partial x^k}}\frac{\partial}{\partial x^i}\rangle \\
& & \,\,+ \langle AX_*\frac{\partial}{\partial x^i},\bar\nabla_{X_*\frac{\partial}{\partial x^k}}X_*\frac{\partial}{\partial x^j}-X_*\nabla_{\frac{\partial}{\partial x^k}}\frac{\partial}{\partial x^j}\rangle.
\end{eqnarray*}
However since $X_*\frac{\partial}{\partial x^i}= \frac{\partial}{\partial x^i}+u_i Y$  we compute
\begin{eqnarray*}
& & \bar\nabla_{X_*\frac{\partial}{\partial x^k}}X_*\frac{\partial}{\partial x^i}-X_*\nabla_{\frac{\partial}{\partial x^k}}\frac{\partial}{\partial x^i}=\bar\nabla_{\frac{\partial}{\partial x^k}}\frac{\partial}{\partial x^i}+u_{i,k} Y+u_i \bar\nabla_{\frac{\partial}{\partial x^k}}Y+u_k\bar\nabla_Y\frac{\partial}{\partial x^i}
+u_i u_k \bar\nabla_Y Y\\
& &\,\,\,\, -\nabla_{\frac{\partial}{\partial x^k}}\frac{\partial}{\partial x^i} -\langle \nabla u, \nabla_{\frac{\partial}{\partial x^k}}\frac{\partial}{\partial x^i}\rangle Y.
\end{eqnarray*}
Therefore
\begin{eqnarray*}
\bar\nabla_{X_*\frac{\partial}{\partial x^k}}X_*\frac{\partial}{\partial x^i}-X_*\nabla_{\frac{\partial}{\partial x^k}}\frac{\partial}{\partial x^i}=u_{i;k} Y+u_i \bar\nabla_{\frac{\partial}{\partial x^k}}Y+u_k\bar\nabla_{\frac{\partial}{\partial x^i}}Y
+u_i u_k \bar\nabla_Y Y.
\end{eqnarray*}
Hence using  (\ref{aij-gamma}),  (\ref{killing-1}) and (\ref{killing-2}) we obtain
\begin{eqnarray*}
& & \bar\nabla_{X_*\frac{\partial}{\partial x^k}}X_*\frac{\partial}{\partial x^i}-X_*\nabla_{\frac{\partial}{\partial x^k}}\frac{\partial}{\partial x^i}=(Wa_{ik}+u_i u_k u^l\frac{\gamma_l}{2\gamma^2}) Y
+\frac{1}{2}u_i u_k \frac{\nabla \gamma}{\gamma^2}\\
& & \,\,=Wa_{ik} Y +\frac{1}{2\gamma^2}u_i u_k (\langle \nabla u, \nabla \gamma\rangle Y
+ \nabla \gamma)=Wa_{ik} Y +\frac{1}{2\gamma^2}u_i u_k X_*\nabla \gamma.
\end{eqnarray*}
Hence it follows that
\begin{eqnarray*}
& & \langle AX_*\frac{\partial}{\partial x^j}, \bar\nabla_{X_*\frac{\partial}{\partial x^k}}X_*\frac{\partial}{\partial x^i}-X_*\nabla_{\frac{\partial}{\partial x^k}}\frac{\partial}{\partial x^i}\rangle =
\langle AX_*\frac{\partial}{\partial x^j}, \frac{u_i u_k}{2\gamma^2}X_*\nabla \gamma + W a_{ik}Y\rangle\\
& &\,\,=\frac{1}{\gamma}W a_{ik}a_j^l u_l +\frac{u_i u_k}{2\gamma^2}a_{jl}\gamma^l.
\end{eqnarray*}
We conclude that
\begin{eqnarray*}
& & \nabla_k b_{ij} =\langle(\nabla^\Sigma_{X_*\frac{\partial}{\partial x^k}} A)X_*\frac{\partial}{\partial x^i},X_*\frac{\partial}{\partial x^j}\rangle +\frac{1}{\gamma}W a_{ik}a_j^l u_l +\frac{u_i u_k}{2\gamma^2}a_{jl}\gamma^l \\
& & \,\,+ \frac{1}{\gamma}W a_{jk}a_i^l u_l +\frac{u_j u_k}{2\gamma^2}a_{il}\gamma^l,
\end{eqnarray*}
that is,
\begin{eqnarray}
\label{nabla-nabla}
\nabla_k b_{ij} =\nabla^\Sigma_k a_{ij} +\frac{1}{\gamma}W a_{ik}a_j^l u_l + \frac{1}{\gamma}W a_{jk}a_i^l u_l +\frac{u_i u_k}{2\gamma^2}a_{jl}\gamma^l  +\frac{u_j u_k}{2\gamma^2}a_{il}\gamma^l.
\end{eqnarray}

Now we use (\ref{nabla-nabla}) for computing the Hessian of the function $\theta$.

\begin{lemma}
\label{hessian-theta}
The trace of the Hessian of $\theta$ in $\Omega$ calculated with respect to the metric in $\Sigma$ is given by
\begin{eqnarray*}
& & g^{ik}\theta_{i;k} =- |A|^2 \theta-2\langle \nabla^2 d, X^* II\rangle_\Sigma - n\langle \nabla^\Sigma H, \nabla^\Sigma d\rangle - nHW\langle AY^T, \nabla^\Sigma d\rangle-{\rm Ric}(\nabla d, \frac{\nabla u}{W}) \\
& &\,\,\,\,-{\rm tr}_\Sigma\nabla_{\frac{\nabla u}{W}} \nabla^2 d - \frac{|\nabla u|^2}{W^2} \langle A \nabla^\Sigma d, X_*\frac{\nabla\gamma}{2\gamma}\rangle +\frac{1}{2}\langle AY^T, Y^T\rangle\langle \nabla d, \nabla \gamma\rangle-\frac{1}{2W^2}\nabla^2 d(\frac{\nabla u}{W},\nabla\gamma)\\
& & \,\,\,\,-\frac{\gamma}{W^2}\langle N, \nabla \kappa\rangle
+\kappa (nH-\gamma\langle AY^T, Y^T\rangle)-\kappa \frac{1}{2W^2}\langle N,  \nabla\gamma\rangle.
\end{eqnarray*}
\end{lemma}

\begin{proof}  Notice  that we may write  (\ref{nabla-theta}) as
\begin{equation}
\theta_i = -g^{jl} b_{il} d_j +(d_{i;j}-\kappa\sigma_{ij})N^j.
\end{equation}
Hence we have
\begin{eqnarray*}
& & g^{ik} \theta_{i;k} = -g^{ik}(g^{jl} b_{il} d_j)_{;k}+g^{ik}(d_{i;jk}-\kappa_k \sigma_{ij})N^j +g^{ik}(d_{i;j}-\kappa \sigma_{ij})N^j_{;k}\\
& & \,\, =-g^{ik}(g^{jl} b_{il} d_j)_{;k}+g^{ik}(d_{i;kj}+R^l_{jki}d_l-\kappa_k \sigma_{ij})N^j -g^{ik}(d_{i;j}-\kappa \sigma_{ij})(a^j_{k}-N_k \frac{\gamma^j}{2\gamma}).
\end{eqnarray*}
However
\begin{eqnarray*}
& & g^{ik}(g^{jl} b_{il} d_j)_{;k} = g^{jl}g^{ik} b_{il;k}d_j +g^{ik} g^{jl}_{;k}b_{il}d_j + g^{ik}g^{jl}b_{il}d_{j;k}\\
& & \,\,=g^{jl}g^{ik}(\nabla^\Sigma_k a_{il}+\frac{1}{\gamma}W a_{ik}a^m_l u_m+\frac{1}{\gamma}W a_{lk}a^m_i u_m+u_i u_k a_{lm}\frac{\gamma^m}{2\gamma^2}
+u_l u_k a_{im}\frac{\gamma^m}{2\gamma^2})d_j \\
&& \,\,\,\,+g^{ik} g^{jl}_{;k}b_{il}d_j + g^{ik}g^{jl}b_{il}d_{j;k}
\end{eqnarray*}
Hence using Codazzi's equation we obtain
\begin{eqnarray*}
& & g^{ik}(g^{jl} b_{il} d_j)_{;k} =g^{jl}(nH_l +n\frac{1}{\gamma}WH a_l^m u_m+\frac{1}{\gamma}W a_l^i a^m_i u_m + \frac{|\nabla u|^2}{W^2} a_{lm}\frac{\gamma^m}{2\gamma}+u_l u_k a^k_m\frac{\gamma^m}{2\gamma^2})d_j \\
& & \,\,\,\,+g^{jl}g^{ik}\langle \bar R(X_*\frac{\partial}{\partial x^i}, X_*\frac{\partial}{\partial x^k})N, X_*\frac{\partial}{\partial x^l}\rangle d_j+g^{ik} g^{jl}_{;k}b_{il}d_j + g^{ik}g^{jl}b_{il}d_{j;k}
\end{eqnarray*}
Using that $g^{jl}u_l =\frac{\gamma}{W^2} u^j$ we conclude that
\begin{eqnarray*}
& & g^{ik}(g^{jl} b_{il} d_j)_{;k} = ng^{jl}H_ld_j - n\frac{1}{\gamma}W^2H g^{jl}a_l^m N_m d_j-\frac{1}{\gamma}W^2 g^{jl}a_l^i a^m_i N_m d_j\\
& & + \frac{|\nabla u|^2}{W^2} a^j_{m}\frac{\gamma^m}{2\gamma}d_j+N^j N_k a^k_m\frac{\gamma^m}{2\gamma}d_j  +g^{ik} g^{jl}_{;k}b_{il}d_j + g^{ik}g^{jl}b_{il}d_{j;k}
\end{eqnarray*}
However we have
\begin{eqnarray*}
g^{jl}_{;k} = (\sigma^{jl}-N^j N^l)_{;k} = - N^j_{;k}N^l - N^j N^l_{;k} = (a^j_k -N_k \frac{\gamma^j}{2\gamma})N^l +N^j (a^l_k -N_k \frac{\gamma^l}{2\gamma}).
\end{eqnarray*}
and
\begin{eqnarray*}
& & \bar\nabla_{\frac{\partial}{\partial x^k}} N = \bar\nabla_{X_*\frac{\partial}{\partial x^k}}N -\bar\nabla_{u_k Y}N=-A X_*\frac{\partial}{\partial x^k} - u_k  \bar\nabla_Y (\frac{\gamma}{W}Y -\frac{\nabla u}{W}) \\
& & \,\, =-A X_*\frac{\partial}{\partial x^k} - \frac{u_k}{2W}\big(\frac{\nabla\gamma}{\gamma} +\langle \nabla u, \frac{\nabla\gamma}{\gamma}\rangle Y)
\end{eqnarray*}
from what follows that
\begin{eqnarray*}
& & g^{ik}(g^{jl} b_{il} d_j)_{;k} = ng^{jl}H_ld_j - n\frac{1}{\gamma}W^2H a_l^m N_m g^{jl}d_j-\frac{1}{\gamma}W^2 a_l^i a^m_i N_m g^{jl}d_j\\
& &\,\,\,\, + \frac{|\nabla u|^2}{W^2} a^j_{m}\frac{\gamma^m}{2\gamma}d_j+N^j N_k a^k_m\frac{\gamma^m}{2\gamma}d_j +a^k_la^j_kN^l d_j\\
& & \,\,\,\,-a^k_lN_k N^l \frac{\gamma^j}{2\gamma} d_j+ a^k_la^l_k N^j d_j -a^k_lN_k \frac{\gamma^l}{2\gamma}N^jd_j +g^{jl}a^k_{l}d_{j;k}
\end{eqnarray*}
Therefore
\begin{eqnarray*}
& &  g^{ik}\theta_{i;k} =- ng^{jl}H_ld_j + n\frac{1}{\gamma}W^2H a_l^m N_m g^{jl}d_j +\frac{1}{\gamma}W^2 a_l^i a^m_i N_m g^{jl}d_j- \frac{|\nabla u|^2}{W^2} a^j_{m}\frac{\gamma^m}{2\gamma}d_j\\
& &\,\,\,\,  -a^k_la^j_kN^l d_j+a^k_lN_k N^l \langle \nabla d, \frac{\nabla \gamma}{2\gamma}\rangle- a^k_la^l_k \theta-g^{jl}a^k_{l}d_{j;k}\\
& & \,\,\,\,+g^{ik}(d_{i;kj}+R^l_{jki}d_l-\kappa_k \sigma_{ij})N^j
-g^{ik}(d_{i;j}-\kappa \sigma_{ij})(a^j_{k}-N_k \frac{\gamma^j}{2\gamma})
\end{eqnarray*}
Now using the fact that  $g^{ij}u_j = \frac{\gamma}{W^2}u^i$
and therefore $g^{ij}N_j = \frac{\gamma}{W^2}N^i$ we obtain
\begin{eqnarray*}
& & a^m_i N_m = g^{km}a_{ik}N_m =\frac{\gamma}{W^2}a_{ik}N^k =\frac{\gamma}{W^2}\langle AX_*\frac{\partial}{\partial x^i}, N^k X_*\frac{\partial}{\partial x^k}\rangle \\
& &\,\,=
\frac{\gamma}{W^2}\langle AX_*\frac{\partial}{\partial x^i}, N^k\frac{\partial}{\partial x^k}+\langle N^k\frac{\partial}{\partial x^k}, \nabla u\rangle Y\rangle\\
& & \,\,=\frac{\gamma}{W^2}\langle AX_*\frac{\partial}{\partial x^i}, N-\frac{\gamma}{W}Y +\langle N, \nabla u\rangle Y\rangle\\
& & \,\,=-\frac{\gamma}{W^2}\langle AX_*\frac{\partial}{\partial x^i}, Y\rangle \big(\frac{\gamma}{W}+\frac{|\nabla u|^2}{W}\big)=-\frac{\gamma}{W}\langle AX_*\frac{\partial}{\partial x^i}, Y\rangle
=-\frac{\gamma}{W}\langle AY^T, X_*\frac{\partial}{\partial x^i}\rangle.
\end{eqnarray*}
Therefore
\begin{eqnarray*}
& & a_l^m N_m g^{jl}d_j = -\frac{\gamma}{W}\langle AY^T, g^{jl}d_jX_*\frac{\partial}{\partial x^l}\rangle=-\frac{\gamma}{W}\langle AY^T, \nabla^\Sigma d\rangle
\end{eqnarray*}
Moreover notice that
\begin{eqnarray*}
a^k_l N^l = g^{km}a_{ml}N^l= - g^{km}W\langle AY^T, X_*\frac{\partial}{\partial x^m}\rangle
\end{eqnarray*}
and
\begin{eqnarray*}
a_{ik}N^k = -W\langle AY^T, X_*\frac{\partial}{\partial x^i}\rangle.
\end{eqnarray*}
Similarly we have
\begin{eqnarray*}
a^j_k d_j = g^{jm}d_j\langle AX_*\frac{\partial}{\partial x^k}, X_*\frac{\partial}{\partial x^m}\rangle=\langle AX_*\frac{\partial}{\partial x^k}, \nabla^\Sigma d\rangle=\langle A \nabla^\Sigma d, X_*\frac{\partial}{\partial x^k}\rangle.
\end{eqnarray*}
Replacing this above we obtain
\begin{eqnarray*}
& &  g^{ik}\theta_{i;k} =- n\langle \nabla^\Sigma H, \nabla^\Sigma d\rangle - nHW\langle AY^T, \nabla^\Sigma d\rangle -W \langle AY^T, A\nabla^\Sigma d\rangle - \frac{|\nabla u|^2}{W^2} \langle A \nabla^\Sigma d, X_*\frac{\nabla\gamma}{2\gamma}\rangle\\
& &\,\,\,\,  +W\langle AY^T, A\nabla^\Sigma d\rangle +\gamma\langle AY^T, Y^T\rangle\langle \nabla d, \frac{\nabla \gamma}{2\gamma}\rangle- |A|^2 \theta-g^{jl}a^k_{l}d_{j;k}\\
& & \,\,\,\,+g^{ik}(d_{i;kj}+R^l_{jki}d_l-\kappa_k \sigma_{ij})N^j
-g^{ik}(d_{i;j}-\kappa \sigma_{ij})(a^j_{k}-N_k \frac{\gamma^j}{2\gamma})
\end{eqnarray*}
Therefore
\begin{eqnarray*}
& &  g^{ik}\theta_{i;k} =- n\langle \nabla^\Sigma H, \nabla^\Sigma d\rangle - nHW\langle AY^T, \nabla^\Sigma d\rangle  - \frac{|\nabla u|^2}{W^2} \langle A \nabla^\Sigma d, X_*\frac{\nabla\gamma}{2\gamma}\rangle\\
& &\,\,\,\,   +\gamma\langle AY^T, Y^T\rangle\langle \nabla d, \frac{\nabla \gamma}{2\gamma}\rangle- |A|^2 \theta-g^{jl}a^k_{l}d_{j;k}\\
& & \,\,\,\,+g^{ik}(d_{i;kj}+R^l_{jki}d_l-\kappa_k \sigma_{ij})N^j
-g^{ik}(d_{i;j}-\kappa \sigma_{ij})(a^j_{k}-N_k \frac{\gamma^j}{2\gamma})
\end{eqnarray*}
However
\begin{eqnarray*}
g^{ik}\sigma_{ij} =g^{ik}(g_{ij}-\frac{u_i u_j}{\gamma})= \delta^k_j -\frac{1}{W^2}u^ku_j =\delta^k_j -N^k N_j.
\end{eqnarray*}
Hence we have
\begin{eqnarray*}
& &  g^{ik}\theta_{i;k} =- n\langle \nabla^\Sigma H, \nabla^\Sigma d\rangle - nHW\langle AY^T, \nabla^\Sigma d\rangle  - \frac{|\nabla u|^2}{W^2} \langle A \nabla^\Sigma d, X_*\frac{\nabla\gamma}{2\gamma}\rangle\\
& &\,\,\,\,  +\frac{1}{2}\langle AY^T, Y^T\rangle\langle \nabla d, \nabla \gamma\rangle- |A|^2 \theta-2g^{ik}g^{jl}d_{i;j}a_{kl}+\frac{1}{2W^2}d_{i;j}N^i \gamma^j\\
& & \,\,\,\,+g^{ik}d_{i;kj}N^j-\textrm{Ric}(\nabla d, \frac{\nabla u}{W})-\frac{\gamma}{W^2}\langle N, \nabla \kappa\rangle
+\kappa (nH-\gamma\langle AY^T, Y^T\rangle)-\kappa \frac{1}{2W^2}\langle N,  \nabla\gamma\rangle
\end{eqnarray*}
This finishes the proof of the Lemma.
\end{proof}

Using Lemma \ref{hessian-theta} we will obtain an expression for $Lh$. Notice that
\begin{equation*}
h_{i;k} = \alpha d_{i;k} -\phi_i \theta_k -\phi_k \theta_i-\phi_{i;k}\theta -\phi \theta_{i;k}.
\end{equation*}
Moreover it holds that
\begin{eqnarray*}
& & 2g^{ik}\phi_i \theta_k = 2g^{ik}\phi_i \langle A\nabla^\Sigma d, X_*\frac{\partial}{\partial x^k}\rangle -2g^{ik}d_{k;l}\phi_i N^l+2\kappa g^{ik}\sigma_{kl}\phi_i N^l\\
& & \,\,=2\langle A\nabla^\Sigma d, \nabla^\Sigma\phi\rangle -2g^{ik}d_{k;l}\phi_i N^l+2\kappa\frac{\gamma}{W^2}\langle \nabla\phi, N\rangle.
\end{eqnarray*}
We conclude that
\begin{eqnarray*}
& & g^{ik}h_{i;k} =\alpha g^{ik} d_{i;k} +2\langle A\nabla^\Sigma d, \nabla^\Sigma\phi\rangle -2g^{ik}d_{k;l}\phi_i N^l+2\kappa\frac{\gamma}{W^2}\langle \nabla\phi, N\rangle-g^{ik}\phi_{i;k}\theta\\
& &\,\,\,\, + n\phi\langle \nabla^\Sigma H, \nabla^\Sigma d\rangle + n\phi HW\langle AY^T, \nabla^\Sigma d\rangle
+ \frac{|\nabla u|^2}{W^2}\phi \langle A \nabla^\Sigma d, X_*\frac{\nabla\gamma}{2\gamma}\rangle\\
& &\,\,\,\,  -\frac{1}{2}\phi\langle AY^T, Y^T\rangle\langle \nabla d, \nabla \gamma\rangle+ |A|^2 \phi\theta+2g^{ik}g^{jl}d_{i;j}a_{kl}\phi-\frac{1}{2W^2}\phi d_{i;j}N^i \gamma^j\\
& & \,\,\,\,-g^{ik}d_{i;kj}N^j \phi+\textrm{Ric}(\nabla d, \frac{\nabla u}{W})\phi+\frac{\gamma}{W^2}\langle N, \nabla \kappa\rangle\phi
-\kappa (nH-\gamma\langle AY^T, Y^T\rangle)\phi\\
& & \,\,\,\,+\kappa \frac{1}{2W^2}\langle N,  \nabla\gamma\rangle\phi.
\end{eqnarray*}
Now we compute the derivatives with respect to $t$. We have
\begin{eqnarray*}
& &  \theta_t =X_*\frac{\partial}{\partial t}\langle N, \bar\nabla d\rangle = \langle \bar\nabla_{X_*\frac{\partial}{\partial t}}N, \bar\nabla d\rangle + \langle N, \bar\nabla_{X_*\frac{\partial}{\partial t}}\bar\nabla d\rangle \\
& & \,\, =-\langle \nabla^\Sigma (nH-\mathcal{H}), \bar\nabla d\rangle+(nH-\mathcal{H})\langle N, \bar\nabla_{N}\bar\nabla d\rangle.
\end{eqnarray*}
However
\begin{eqnarray*}
\langle N, \bar\nabla_{N}\bar\nabla d\rangle =-\frac{1}{2W^2}\langle \bar\nabla\gamma, \bar\nabla d\rangle+\langle  \frac{\nabla u}{W}, \bar\nabla_{\frac{\nabla u}{W}}\bar\nabla d\rangle.
\end{eqnarray*}
Hence we have
\begin{eqnarray*}
& &  \theta_t =
-\langle \nabla^\Sigma (nH-\mathcal{H}), \bar\nabla d\rangle+(nH-\mathcal{H})(-\frac{1}{2W^2}\langle \nabla\gamma, \nabla d\rangle+\langle  \frac{\nabla u}{W}, \bar\nabla_{\frac{\nabla u}{W}}\bar\nabla d\rangle).
\end{eqnarray*}
Moreover we have
\begin{equation}
d_t = \langle X_* \frac{\partial}{\partial t}, \bar\nabla d\rangle = (nH-\mathcal{H})\langle N,\bar\nabla d\rangle = (nH-\mathcal{H})\theta.
\end{equation}
Therefore
\begin{eqnarray*}
& & h_t = \alpha (nH-\mathcal{H})\theta -(nH-\mathcal{H})\langle N, \bar\nabla \phi\rangle \theta+\phi\langle \nabla^\Sigma (nH-\mathcal{H}), \bar\nabla d\rangle\\
& &\,\,\,\,-\phi(nH-\mathcal{H})(-\frac{1}{2W^2}\langle \nabla\gamma, \nabla d\rangle+\langle  \frac{\nabla u}{W}, \bar\nabla_{\frac{\nabla u}{W}}\bar\nabla d\rangle)
\end{eqnarray*}
We also compute
\begin{eqnarray*}
\langle \nabla \gamma, \nabla h\rangle =\alpha\langle\nabla d, \nabla \gamma\rangle +\phi\langle A\nabla^\Sigma d, X_*\nabla\gamma\rangle -\langle \nabla\phi,\nabla \gamma\rangle \theta -\phi d_{i;j}\gamma^i N^j +\kappa\phi \langle N, \nabla \gamma\rangle.
\end{eqnarray*}
Now we obtain
\begin{eqnarray*}
g^{ik}d_{i;k}= \Delta d -\langle \nabla_{\frac{\nabla u}{W}}\nabla d, \frac{\nabla u}{W}\rangle = -nH_d -\langle \nabla_{\frac{\nabla u}{W}}\nabla d, \frac{\nabla u}{W}\rangle
\end{eqnarray*}
and
\begin{eqnarray*}
& & g^{ik}d_{k;l}\phi_i N^l = d_{k;l}\phi^k N^l -d_{k;l}N^k N^l N^i \phi_i= -\langle \nabla_{\frac{\nabla u}{W}}\nabla d, \nabla \phi\rangle -
\langle \nabla_{\frac{\nabla u}{W}}\nabla d, \frac{\nabla u}{W}\rangle\langle N,\nabla \phi\rangle.
\end{eqnarray*}
Moreover we have
\begin{eqnarray*}
g^{ij}\phi_{i;j} = \Delta \phi -\langle \nabla_{\frac{\nabla u}{W}}\nabla \phi, \frac{\nabla u}{W}\rangle
\end{eqnarray*}
and
\begin{eqnarray*}
& & g^{ik}d_{i;kj}N^j = (\sigma^{ik}d_{i;k})_{;j} N^j -d_{i;kj}N^i N^k N^j  = -n(H_d)_j N^j +\nabla^3 d(\frac{\nabla u}{W},\frac{\nabla u}{W},\frac{\nabla u}{W}).
\end{eqnarray*}
Therefore grouping and rearranging these expressions we obtain
\begin{eqnarray*}
& &  Lh =|A|^2 \phi\theta + n\phi HW\langle AY^T, \nabla^\Sigma d\rangle
+\big(\kappa\gamma-\frac{1}{2}\langle \nabla d, \nabla \gamma\rangle\big)\phi\langle AY^T, Y^T\rangle\\
& &\,\,\,\, +2\langle A\nabla^\Sigma d, \nabla^\Sigma\phi\rangle +2\langle A, \nabla^2 d\rangle_\Sigma\phi-\frac{1}{W^2}\phi\langle A\nabla^\Sigma d, X_*\nabla\gamma\rangle\nonumber\\
& & \,\,\,\,\big(nH-\mathcal{H}\big)\big(\langle N, \nabla \phi\rangle \theta-\alpha\theta-\frac{1}{2W^2}\langle \nabla\gamma, \nabla d\rangle\phi+\langle \nabla_{\frac{\nabla u}{W}}\nabla d, \frac{\nabla u}{W}\rangle\phi\big)-n\kappa H\phi\nonumber\\
& & \,\,\,\,-n\alpha H_d +\big(2\langle N, \nabla\phi\rangle -\alpha\big)\langle \nabla_{\frac{\nabla u}{W}}\nabla d, \frac{\nabla u}{W}\rangle +2\langle \nabla_{\frac{\nabla u}{W}}\nabla d, \nabla \phi\rangle
\nonumber\\
& &\,\,\,\,-\phi \langle \nabla_{\frac{\nabla u}{W}}\nabla d, \frac{\nabla \gamma}{2\gamma}\rangle +\phi\langle \nabla^\Sigma \mathcal{H}, \bar\nabla d\rangle+n\langle \nabla H_d, N\rangle \phi-\phi\nabla^3 d(\frac{\nabla u}{W},\frac{\nabla u}{W},\frac{\nabla u}{W}) +\textrm{Ric}(\nabla d, \frac{\nabla u}{W})\phi\nonumber\\
& & \,\,\,\,+\frac{\gamma}{W^2}\langle N, \nabla \kappa\rangle\phi
-\big(\frac{1}{2\gamma}+\frac{1}{2W^2}\big)\alpha\langle\nabla d, \nabla \gamma\rangle  +\big(\frac{1}{2\gamma}+\frac{1}{2W^2}\big)\langle \nabla\phi,\nabla \gamma\rangle \theta\nonumber \\
& & \,\,\,\, -\kappa\phi \langle N, \frac{\nabla \gamma}{2\gamma}\rangle+2\kappa\frac{\gamma}{W^2}\langle \nabla\phi, N\rangle-\big(\Delta \phi -\langle \nabla_{\frac{\nabla u}{W}}\nabla \phi, \frac{\nabla u}{W}\rangle\big)\theta.\nonumber
\end{eqnarray*}

\begin{lemma}\label{LW-lemma} We have
\begin{eqnarray*}
& & LW-\frac{2}{W}g^{ij}W_i W_j = |A|^2 W +nHW^3\langle AY^T,  Y^T\rangle -nHW^3\langle \frac{\nabla \gamma}{2\gamma^2}, N\rangle-3\gamma\langle A Y^T, X_*\frac{\nabla\gamma}{2\gamma}\rangle
\\
& &\,\,\,\,+ g^{ij}\frac{\gamma_{i;j}}{2\gamma}W - \frac{3}{4}\frac{|\nabla\gamma|^2}{4\gamma^2}W  -\frac{1}{4}\langle \frac{\nabla \gamma}{2\gamma}, N\rangle^2 W+ \gamma W\langle \bar\nabla_{N}\frac{\bar\nabla \gamma}{2\gamma^2}, N\rangle - W\langle \nabla^\Sigma \mathcal{H}, N\rangle \\
& &\,\,\,\,- \frac{|\nabla \gamma|^2}{4\gamma}\frac{1}{W}-W_t.
\end{eqnarray*}
\end{lemma}

\begin{proof} Notice that
\begin{eqnarray*}
& & W_i = -W^2\big( \langle \bar\nabla_{X_*\frac{\partial}{\partial x^i}} Y, N\rangle +\langle Y, \bar\nabla_{X_*\frac{\partial}{\partial x^i}}N\rangle\big)\\
& & \,\,= -W^2\big( \langle \bar\nabla_{\frac{\partial}{\partial x^i}} Y, N\rangle+u_i\langle \bar\nabla_{Y} Y, N\rangle -\langle Y, AX_*\frac{\partial}{\partial x^i}\rangle\big)\\
& & \,\,= -W^2\big( -\frac{\gamma_i}{2\gamma}\langle Y, N\rangle+u_i\langle \frac{\nabla \gamma}{2\gamma^2}, N\rangle -\langle Y, AX_*\frac{\partial}{\partial x^i}\rangle\big).
\end{eqnarray*}
Therefore
\begin{eqnarray*}
& & W_i =\frac{\gamma_i}{2\gamma}W+N_iW^3\langle \frac{\nabla \gamma}{2\gamma^2}, N\rangle +W^2\langle A Y^T, X_*\frac{\partial}{\partial x^i}\rangle.
\end{eqnarray*}
However
\begin{eqnarray*}
& & \langle A Y^T, X_*\frac{\partial}{\partial x^i}\rangle =g^{kl}\langle Y, X_*\frac{\partial}{\partial x^k}\rangle\langle  X_*\frac{\partial}{\partial x^l}, A X_*\frac{\partial}{\partial x^i}\rangle =
g^{kl}\langle Y, u_k Y\rangle b_{il}=\frac{1}{W^2}u^l b_{il}.
\end{eqnarray*}
Hence it follows that
\begin{eqnarray*}
& & W_i =\frac{\gamma_i}{2\gamma}W+N_iW^3\langle \frac{\nabla \gamma}{2\gamma^2}, N\rangle - WN^l b_{il}.
\end{eqnarray*}
Hence we obtain
\begin{eqnarray*}
& &\frac{1}{W} g^{ij}W_i W_j =\frac{|\nabla^\Sigma \gamma|^2}{4\gamma^2}W+W\langle \nabla\gamma, N\rangle \langle \frac{\nabla \gamma}{2\gamma^2}, N\rangle +\langle A Y^T, \frac{\nabla^\Sigma\gamma}{\gamma}\rangle W^2\\
& &\,\,\,\,+\gamma|\nabla u|^2 \langle \frac{\nabla \gamma}{2\gamma^2}, N\rangle^2 W-\langle A Y^T, Y^T\rangle \langle \frac{\nabla \gamma}{\gamma}, N\rangle W^3+ \langle A Y^T, A Y^T\rangle W^3
\end{eqnarray*}
Now we compute
\begin{eqnarray*}
& & W_{i;j} = \big(\frac{\gamma_{i;j}}{2\gamma}-\frac{\gamma_i \gamma_j}{2\gamma^2}\big)W+\frac{\gamma_i}{2\gamma}W_j+N_{i;j}W^3\langle \frac{\nabla \gamma}{2\gamma^2}, N\rangle+3N_{i}W^2W_j\langle \frac{\nabla \gamma}{2\gamma^2}, N\rangle\\
& & \,\,\,\,+N_{i}W^3\big(\langle \bar\nabla_{X_*\frac{\partial}{\partial x^j}}\frac{\bar\nabla \gamma}{2\gamma^2}, N\rangle-\langle \frac{\nabla \gamma}{2\gamma^2}, AX_*\frac{\partial}{\partial x^j}\rangle\big) - W_j N^l b_{il} - WN^l_{;j} b_{il} - WN^l b_{il;j}.
\end{eqnarray*}
However we have
\begin{eqnarray*}
& & g^{ij}\frac{\gamma_i}{2\gamma}W_j = \frac{|\nabla^\Sigma\gamma|^2}{4\gamma^2}W+\langle \frac{\nabla \gamma}{2\gamma}, N\rangle^2 W+W^2\langle A Y^T, \frac{\nabla^\Sigma\gamma}{2\gamma}\rangle
\end{eqnarray*}
and
\begin{eqnarray*}
& & g^{ij}N_{i;j} = g^{ij}\sigma_{ik}N^k_{;j}= -(\delta^j_k -N^j N_k)(a^k_{j}-N_j \frac{\gamma^k}{2\gamma})\\
& & \,\,= -nH+\frac{\gamma}{W^2}\langle N,\frac{\nabla\gamma}{2\gamma}\rangle+\gamma\langle AY^T, Y^T\rangle.
\end{eqnarray*}
Moreover we compute
\begin{eqnarray*}
g^{ij}N_i W_j =\frac{\gamma}{W}\langle N,\frac{\nabla\gamma}{2\gamma}\rangle  +\frac{\gamma|\nabla u|^2}{W} \langle \frac{\nabla \gamma}{2\gamma^2}, N\rangle-\gamma W\langle A Y^T, Y^T\rangle
\end{eqnarray*}
and
\begin{eqnarray*}
& & g^{ij}N_{i}W^3\big(\langle \bar\nabla_{X_*\frac{\partial}{\partial x^j}}\frac{\bar\nabla \gamma}{2\gamma^2}, N\rangle-\langle \frac{\bar\nabla \gamma}{2\gamma^2}, AX_*\frac{\partial}{\partial x^j}\rangle\big)
\\
& & \,\,=\gamma W\big(\langle \bar\nabla_{N-WY}\frac{\bar\nabla \gamma}{2\gamma^2}, N\rangle-\langle A\frac{\nabla^\Sigma \gamma}{2\gamma^2}, -W Y\rangle\big)\\
& &\,\, =\gamma W\langle \bar\nabla_{N}\frac{\bar\nabla \gamma}{2\gamma^2}, N\rangle  + W\frac{|\nabla \gamma|^2}{4\gamma^2} + \gamma W^2
\langle A\frac{\nabla^\Sigma \gamma}{2\gamma^2}, Y^T\rangle.
\end{eqnarray*}
We also have
\begin{eqnarray*}
& & 2Wg^{ij}W_j\langle A Y^T, X_*\frac{\partial}{\partial x^i}\rangle  = 2W^2\langle A Y^T, \frac{\nabla^\Sigma\gamma}{2\gamma}\rangle - W^3\langle \frac{\nabla \gamma}{\gamma}, N\rangle\langle A Y^T, Y^T\rangle +2W^3 \langle A Y^T, AY^T\rangle.
\end{eqnarray*}
Now we compute
\begin{eqnarray*}
& & g^{ij}WN^l b_{il;j} = WN^l g^{ij}\nabla^\Sigma_j a_{il}+\frac{1}{\gamma}W^2 g^{ij} a_{ij}a^m_l N^l u_m +\frac{1}{\gamma}W^2 g^{ij}  a_{lj} N^l a^m_i u_m \\
& & \,\,\,\,+W g^{ij}\frac{u_i u_j}{2\gamma^2}a_{lm}N^l\gamma^m +W g^{ij} N^l\frac{u_l u_j}{2\gamma^2}a_{im}\gamma^m.
\end{eqnarray*}
Hence we have
\begin{eqnarray*}
& & g^{ij}WN^l b_{il;j} = WN^l (n\nabla^\Sigma_l H+g^{ij}\langle \bar R(X_*\frac{\partial}{\partial x^i}, X_*\frac{\partial}{\partial x^j})N, X_*\frac{\partial}{\partial x^l}\rangle)+nHW^2 \langle AY^T, N^k X_*\frac{\partial}{\partial x^k}\rangle \\
& & \,\,\,\,+W^2 g^{ij} \langle AY^T, X_*\frac{\partial}{\partial x^j}\rangle\big(-W\langle AY^T, X_*\frac{\partial}{\partial x^i}\rangle\big) \\
& & \,\,\,\,-|\nabla u|^2\langle AY^T, X_*\frac{\nabla \gamma}{2\gamma}\rangle + \frac{|\nabla u|^2}{2\gamma W}(-W\langle AY^T, X_*\frac{\partial}{\partial x^m}\rangle)\gamma^m.
\end{eqnarray*}
Therefore
\begin{eqnarray*}
& & g^{ij}WN^l b_{il;j} = nWN^l\nabla^\Sigma_l H - nHW^3 \langle AY^T, Y^T\rangle -W^3  \langle AY^T,  AY^T\rangle-|\nabla u|^2\langle AY^T, X_*\frac{\nabla \gamma}{\gamma}\rangle.
\end{eqnarray*}
Moreover
\begin{eqnarray*}
& & g^{ij}W_j N^l b_{il} =-W^2\langle AY^T, \frac{\nabla^\Sigma \gamma}{2\gamma}\rangle + W^3\langle \frac{\nabla \gamma}{2\gamma}, N\rangle \langle AY^T, Y^T\rangle -W^3\langle AY^T, AY^T\rangle
\end{eqnarray*}
and
\begin{eqnarray*}
& & W g^{ij} N^l_{;j} b_{il} = -Wg^{ij}(a^l_j-N_j \frac{\gamma^l}{2\gamma}) a_{il} = -|A|^2W - \frac{1}{2}\langle AY^T, X_*\nabla \gamma\rangle.
\end{eqnarray*}
We conclude that
\begin{eqnarray*}
& & g^{ij}W_{i;j} =  |A|^2 W + 2W^3  \langle AY^T,  AY^T\rangle +\big(nH-3\langle \frac{\nabla \gamma}{2\gamma}, N\rangle\big)W^3\langle AY^T,  Y^T\rangle  \\
& & \,\,\,\,+3W^2\langle A Y^T, \frac{\nabla^\Sigma\gamma}{2\gamma}\rangle+|\nabla u|^2\langle AY^T, X_*\frac{\nabla \gamma}{\gamma}\rangle + \frac{1}{2}\langle AY^T, X_*\nabla \gamma\rangle \\
& &\,\,\,\,+ g^{ij}\frac{\gamma_{i;j}}{2\gamma}W - \frac{|\nabla^\Sigma\gamma|^2}{4\gamma^2}W + \frac{|\nabla \gamma|^2}{4\gamma^2}W +\big(5W+3\frac{W}{\gamma}|\nabla u|^2\big)\langle \frac{\nabla \gamma}{2\gamma}, N\rangle^2  \\
& & \,\,\,\, -nHW^3\langle \frac{\nabla \gamma}{2\gamma^2}, N\rangle + \gamma W\langle \bar\nabla_{N}\frac{\bar\nabla \gamma}{2\gamma^2}, N\rangle -nWN^l \nabla^\Sigma_l H.
\end{eqnarray*}
Now
\begin{eqnarray*}
& & \langle \nabla\gamma, \nabla W\rangle = \frac{|\nabla \gamma|^2}{2\gamma}W + \frac{1}{2\gamma^2}\langle \nabla\gamma, N\rangle^2 W^3 + W^2\langle AY^T, X_*\nabla\gamma\rangle.
\end{eqnarray*}
Hence
\begin{eqnarray*}
& & LW-\frac{2}{W}g^{ij}W_i W_j = |A|^2 W +\big(nH+\langle \frac{\nabla \gamma}{2\gamma}, N\rangle\big)W^3\langle AY^T,  Y^T\rangle  \\
& & \,\,\,\,-W^2\langle A Y^T, \frac{\nabla^\Sigma\gamma}{2\gamma}\rangle+|\nabla u|^2\langle AY^T, X_*\frac{\nabla \gamma}{\gamma}\rangle + \frac{1}{2}\langle AY^T, X_*\nabla \gamma\rangle \\
& & \,\,\,\, -\big(\frac{1}{2\gamma}+\frac{1}{2W^2}\big)W^2\langle AY^T, X_*\nabla\gamma\rangle
\\
& &\,\,\,\,+ g^{ij}\frac{\gamma_{i;j}}{2\gamma}W - \frac{3}{4}\frac{|\nabla^\Sigma\gamma|^2}{4\gamma^2}W + \frac{|\nabla \gamma|^2}{4\gamma^2}W +\big(5W+3\frac{W}{\gamma}|\nabla u|^2\big)\langle \frac{\nabla \gamma}{2\gamma}, N\rangle^2  \\
& & \,\,\,\, -nHW^3\langle \frac{\nabla \gamma}{2\gamma^2}, N\rangle + \gamma W\langle \bar\nabla_{N}\frac{\bar\nabla \gamma}{2\gamma^2}, N\rangle -nWN^l \nabla^\Sigma_l H \\
& &\,\,\,\,-\big(\frac{1}{2\gamma}+\frac{1}{2W^2}\big) \big(\frac{|\nabla \gamma|^2}{2\gamma}W + \frac{1}{2\gamma^2}\langle \nabla\gamma, N\rangle^2 W^3\big) -\frac{1}{\gamma^2}\langle \nabla\gamma, N\rangle^2 W -2\gamma|\nabla u|^2 \langle \frac{\nabla \gamma}{2\gamma^2}, N\rangle^2 W-W_t.
\end{eqnarray*}
However
\begin{eqnarray*}
3\frac{W}{\gamma}|\nabla u|^2\langle \frac{\nabla \gamma}{2\gamma}, N\rangle^2 -2\gamma|\nabla u|^2 \langle \frac{\nabla \gamma}{2\gamma^2}, N\rangle^2 W =\frac{W}{\gamma}|\nabla u|^2\langle \frac{\nabla \gamma}{2\gamma}, N\rangle^2
\end{eqnarray*}
and
\begin{eqnarray*}
5W \langle \frac{\nabla \gamma}{2\gamma}, N\rangle^2 -\frac{1}{\gamma^2}\langle \nabla\gamma, N\rangle^2 W =  \langle \frac{\nabla\gamma}{2\gamma}, N\rangle^2 W
\end{eqnarray*}
and
\begin{eqnarray*}
& & -W^2\langle A Y^T, \frac{\nabla^\Sigma\gamma}{2\gamma}\rangle + |\nabla u|^2\langle AY^T, X_*\frac{\nabla \gamma}{\gamma}\rangle + \frac{1}{2}\langle AY^T, X_*\nabla \gamma\rangle -\big(\frac{1}{2\gamma}+\frac{1}{2W^2}\big)W^2\langle AY^T, X_*\nabla\gamma\rangle\\
& & \,\,=-3\gamma\langle A Y^T, X_*\frac{\nabla\gamma}{2\gamma}\rangle -W^3\langle\frac{\nabla \gamma}{2\gamma},N\rangle \langle  A Y^T, Y^T\rangle.
\end{eqnarray*}
Moreover we compute
\begin{eqnarray*}
& & \big(\frac{1}{2\gamma}+\frac{1}{2W^2}\big) \big(\frac{|\nabla \gamma|^2}{2\gamma}W + \frac{1}{2\gamma^2}\langle \nabla\gamma, N\rangle^2 W^3\big) \\
& & \,\, = \frac{|\nabla \gamma|^2}{4\gamma^2}W + \frac{1}{\gamma}W^3\langle \frac{\nabla\gamma}{2\gamma}, N\rangle^2 + \frac{|\nabla \gamma|^2}{4\gamma}\frac{1}{W}+ \langle \frac{\nabla\gamma}{2\gamma}, N\rangle^2 W
\end{eqnarray*}
and
\begin{eqnarray*}
& & -nWN^l \nabla^\Sigma_l H  =  -nW\langle\nabla^\Sigma H, N\rangle = - W\langle \nabla^\Sigma \mathcal{H}, N\rangle.
\end{eqnarray*}
We conclude that
\begin{eqnarray*}
& & LW-\frac{2}{W}g^{ij}W_i W_j = |A|^2 W +nHW^3\langle AY^T,  Y^T\rangle -3\gamma\langle A Y^T, X_*\frac{\nabla\gamma}{2\gamma}\rangle
\\
& &\,\,\,\,+ g^{ij}\frac{\gamma_{i;j}}{2\gamma}W - \frac{3}{4}\frac{|\nabla^\Sigma\gamma|^2}{4\gamma^2}W  +\frac{|\nabla u|^2}{\gamma}\langle \frac{\nabla \gamma}{2\gamma}, N\rangle^2 W\\
& & \,\,\,\, -nHW^3\langle \frac{\nabla \gamma}{2\gamma^2}, N\rangle + \gamma W\langle \bar\nabla_{N}\frac{\bar\nabla \gamma}{2\gamma^2}, N\rangle - W\langle \nabla^\Sigma \mathcal{H}, N\rangle \\
& &\,\,\,\,- \frac{1}{\gamma}W^3\langle \frac{\nabla\gamma}{2\gamma}, N\rangle^2 - \frac{|\nabla \gamma|^2}{4\gamma}\frac{1}{W}-W_t.
\end{eqnarray*}
However
\begin{eqnarray*}
\frac{|\nabla u|^2}{\gamma}\langle \frac{\nabla \gamma}{2\gamma}, N\rangle^2 W - \frac{1}{\gamma}W^3\langle \frac{\nabla\gamma}{2\gamma}, N\rangle^2= -\langle \frac{\nabla \gamma}{2\gamma}, N\rangle^2 W
\end{eqnarray*}
and
\begin{eqnarray*}
& & - \frac{3}{4}\frac{|\nabla^\Sigma\gamma|^2}{4\gamma^2}W  -\langle \frac{\nabla \gamma}{2\gamma}, N\rangle^2 W = - \frac{3}{4}\frac{|\nabla\gamma|^2}{4\gamma^2}W + \frac{3}{4}\langle\frac{\nabla\gamma}{2\gamma}, N\rangle^2W -\langle \frac{\nabla \gamma}{2\gamma}, N\rangle^2 W \\
& & \,\,= - \frac{3}{4}\frac{|\nabla\gamma|^2}{4\gamma^2}W  -\frac{1}{4}\langle \frac{\nabla \gamma}{2\gamma}, N\rangle^2 W.
\end{eqnarray*}
Hence we obtain
\begin{eqnarray*}
& & LW-\frac{2}{W}g^{ij}W_i W_j = |A|^2 W +nHW^3\langle AY^T,  Y^T\rangle -nHW^3\langle \frac{\nabla \gamma}{2\gamma^2}, N\rangle-3\gamma\langle A Y^T, X_*\frac{\nabla\gamma}{2\gamma}\rangle
\\
& &\,\,\,\,+ g^{ij}\frac{\gamma_{i;j}}{2\gamma}W - \frac{3}{4}\frac{|\nabla\gamma|^2}{4\gamma^2}W  -\frac{1}{4}\langle \frac{\nabla \gamma}{2\gamma}, N\rangle^2 W+ \gamma W\langle \bar\nabla_{N}\frac{\bar\nabla \gamma}{2\gamma^2}, N\rangle - W\langle \nabla^\Sigma \mathcal{H}, N\rangle \\
& &\,\,\,\,- \frac{|\nabla \gamma|^2}{4\gamma}\frac{1}{W}-W_t.
\end{eqnarray*}
This finishes the proof of the lemma.
\end{proof}

\end{document}